\documentclass[12pt]{amsart}
\usepackage{pdfsync}
\usepackage{graphicx}
\usepackage{enumerate}
\usepackage{amssymb,amsmath,latexsym,amsfonts,amsthm, stmaryrd, MnSymbol}
\usepackage{epstopdf}
\usepackage[mathscr]{eucal}

\DeclareMathOperator{\dom}{dom}

\DeclareMathOperator{\aut}{Aut}

\DeclareMathOperator{\emb}{Emb}

\DeclareMathOperator{\id}{Id}

\newcommand{\N}{\mathbb{N}}

\newcommand{\tr}{\, |\, }

\def\Power #1 { \powerset(#1) }
\def\Bidom #1 { {\mathfrak P} (#1) }

\newtheorem{definition}{{\bf Definition}}[section]

\newtheorem{note}[definition]{{\bf Note}}

\newtheorem{example}[definition]{\noindent {\bf Example}}

\newtheorem{fact}[definition]{\noindent {\bf Fact}}

\def\proofref #1 {{\noindent  {\bf Proof} (#1).}\ }

\newcommand{\overG}{\overline{\mathrm{G}}}
\newcommand{\finp}[1]{\powerset_{<\omega}(#1)}
\newcommand{\llg}{\llbracket}
\newcommand{\rrg}{\rrbracket} 

\newcommand{\restrict}[2]{#1\mspace{-3mu}\mathbin{\restriction}\mspace{-2mu} #2}

\newtheorem{thm}{Theorem}[section]
\newtheorem{lem}{Lemma}[section]
\newtheorem{cor}{Corollary}[section] 

\newtheorem{defin}{Definition}[section]
\newtheorem{problem}{Problem}[section]

\def\endproof{\hfill {\kern 6pt\penalty 500
\raise -0pt\hbox{\vrule \vbox to5pt {\hrule width 5pt
\vfill\hrule}\vrule}}}
\def\centerpicture #1 by #2 (#3){\leavevmode
        \vbox to #2{
        \hrule width #1 height 0pt depth 0pt
        \vfill
        \special{pictfile #3}}}
\begin{document}

\title[The poset of copies]{The poset of copies for automorphism groups of countable relational structures}

\author[C.Laflamme]{Claude Laflamme*} 
\address{Mathematics \& Statistics Department, University of Calgary, Calgary, Alberta, Canada T2N 1N4}
\email{laflamme@ucalgary.ca} 
\thanks{*Supported by NSERC of Canada Grant \# 690404} 
\author[M.Pouzet]{Maurice Pouzet**}\thanks{**Research started while the author visited the Mathematics and Statistics Department of the University of 
Calgary in 2010; the support provided is gratefully acknowledged.}
 \address{ICJ, Math\'ematiques, Universit\'e Claude-Bernard Lyon1, 43 bd. 11 Novembre 1918, 69622 Villeurbanne Cedex, France and Mathematics \& Statistics Department, University of Calgary, Calgary, Alberta, Canada T2N 1N4}
 \email{pouzet@univ-lyon1.fr }
 \author [N.Sauer]{Norbert Sauer*}\thanks{***Supported by NSERC of Canada Grant \# 691325}\address{
Mathematics \& Statistics Department, University of Calgary, Calgary, T2N1N4, Alberta, Canada  T2N1N4}
\email{nsauer@math.ucalgary.ca}
\author[R.Woodrow]{Robert Woodrow*} \address{Mathematics \& Statistics Department, University of Calgary, Calgary, Alberta, Canada T2N 1N4}
\email{woodrow@ucalgary.ca }  
\thanks{The first and fourth author were supported by  by the LABEX MILYON (ANR-10-LABX-0070) of Universit\'e de Lyon within the program ``Investissements d'Avenir (ANR-11-IDEX-0007)" operated by the French National Research Agency (ANR)} 

\date{\today }


\dedicatory{\dagger Dedicated to the memory of Ivo G. Rosenberg}

\begin{abstract}  
Let $\mathrm{G}$  be  a subgroup of the symmetric  group $\mathfrak S(U)$ of all permutations of a countable set $U$. Let $\overline{\mathrm{G}}$ be the topological closure of $\mathrm{G}$ in the function topology on $U^U$.  We initiate the study of the poset $\overline{\mathrm{G}}[U]:=\{f[U]\mid f\in \overline{\mathrm{G}}\}$ of images of the functions in $\overline{\mathrm{G}}$,  being ordered under inclusion. This set $\overline{\mathrm{G}}[U]$ of subsets of the set $U$ will be called  the \emph{poset of copies for} the group  $\mathrm{G}$.  A denomination being justified by the fact that for every subgroup $\mathrm{G}$ of the symmetric group $\mathfrak S(U)$ there exists  a homogeneous relational structure $R$ on $U$ such   that $\overline G$ is the set of embeddings of the homogeneous structure $R$ into itself and  $\overline{\mathrm{G}}[U]$ is the set of copies of $R$ in  $R$ and that the set of bijections $\overline G\cap \mathfrak S(U)$ of $U$ to $U$ forms the group of  automorphisms  of $\mathrm{R}$.                                    
\end{abstract}  
\maketitle

\section{Introduction}

Countable sets will be infinite. For a countable  set $U$ let $\mathfrak{S}(U)$ be the symmetric group of $U$ and let $U^{U}$ be the set of functions of $U$ into $U$  and let $\powerset(U)$ be the power set  of $U$ and let $\finp{U}$  be  the set of finite subsets of $U$.   If $K$ is a subset of $U^U$ then $\overline{K}$ denotes the adherence of $K$ with respect to the function topology, the set $U$ being equipped with the discrete topology.   Then $\overline{K}$ is  the set of functions $f\in U^{U}$  so that for every set $F\in \finp{U}$ there exists a function $g\in K$ for which $f(x)=g(x)$ for all $x\in F$.    That is $\restrict{f}{F}=\restrict{g}{F}$.  For a function $f$ let $f[A]=\{f(x)\mid x\in \dom(f)\cap A\}$. For a set of functions $K$ let $K[A]=\{f[A]\mid f\in K\}$. For two functions $f$ and $g$ let the {\em   composition} of $f$ and $g$ be the function $f\circ g$ with $\dom(f\circ g)=\{x\in \dom(g)\mid g(x)\in \dom(f)\}$ and so that  if $x\in \dom(f\circ g)$ then $f\circ g(x)=f(g(x))$.

Let $\mathrm{G}$  be  a subgroup of the  symmetric group $\mathfrak{S}(U)$  of a countable  set $U$.   Then  $\overG$ is the set of  functions adherent to $\mathrm{G}$ and $\overG[U]$ is the set of images of the functions  in $\overG$.   It follows that $\overG$ is a set of injections of $U$ into $U$ closed under composition and that $\overG[U]\subseteq \powerset(U)$ and that $(\overG[U]; \subseteq)$ is a partially ordered set  under  the inclusion order 
 $\subseteq$. For a simple minded example, if $\mathrm{G}:= \mathfrak S(U)$ then $
 \overG$  is the set of one-to-one maps from $U$ to $U$.  In this case, $\overline{\mathrm{G}}[U]$ is the set of subsets of $U$ with the same cardinality as $U$.  The set of functions   $f\in \overG$ with $f[U]=U$, that is with $f\in  \mathfrak{S}(U)$ forms  a subgroup of $\mathfrak{S}(U)$. The group $\mathrm{G}$ is {\em closed} if $\mathrm{G}=\overG\cap \mathfrak{S}(U)$. The set of images $\overG[U]=\{g[U]\mid g\in \overG \}$ of the functions in $\overG$ will be called the {\em set of copies for} the group $\mathrm{G}$. For a more detailed background of this and towards some established notions  discussed in the next paragraph and further on    we refer the reader to the first about forty pages of \cite{cameronbook}.

The notions for subgroups of $\mathfrak{S}(U)$  investigated in this article are closely related to notions arising in the study of automorphism groups of relational structures. It is not difficult to verify that  $\aut(\mathrm{R})$ the group of automorphisms   of a relational structures $\mathrm{R}$  with countable domain $U$ is closed in $\mathfrak{S}(U)$  for the function topology. Let $\emb(\mathrm{R})$ be  the set of isomorphisms of $\mathrm{R}$ to induced substructures of $\mathrm{R}$.  Define a {\em copy  of }  $\mathrm{R}$ to be a subset $C$ of $U$ for which the  structure $\restrict{\mathrm{R}}{C}$ is isomorphic to $\mathrm{R}$, that is the range of a member of $\emb(\mathrm{R})$. Then  every  copy {\em for} the group $\aut(\mathrm{R})$ is a copy {\em of} $\mathrm{R}$. That is $\overline{\aut(\mathrm{R})}[U]$ is a subset of the set of copies of $\mathrm{R}$.   (A function with domain $U$ adherent to $\aut(\mathrm{R})$ is by definition an isomorphism.)  Let for example $\mathbb{N}:=(N;\leq)$ be the binary relational structure on the set of natural numbers with $\leq$ the natural order. The automorphism group, say $\mathrm{G}$, of $\mathbb{N}$ contains only the identity map on $N$. Hence the set of copies for the group $\mathrm{G}$ is the set $\{N\}$. On the other hand the set of copies of $\mathbb{N}$ has $2^{\aleph_0}$ elements. 

A  relational structure $\mathrm{R}$  is \emph{homogeneous}, see \cite{fraisse54} or \cite{Fra}, if every  isomorphism between finite induced substructures  of $\mathrm{R}$  extends to an automorphism of $\mathrm{R}$. It follows that if $\mathrm{R}$ is homogeneous then $\overline{\aut(\mathrm{R})}[U]$ is equal to  the set of copies of $\mathrm{R}$. It is well known, using the standard back and forth argument, that the order structure $\mathbb{Q}$ of the set of rational numbers  $Q$ is homogeneous. Implying that every copy of $\mathbb{Q}$ is also a copy for the automorphism group of $\mathbb{Q}$.  Let $\mathrm{G}$ be a subgroup of $\mathfrak{S}(U)$ for a countable set $U$. Then the  canonical structure determined by $\mathrm{G}$, that is the relational structure with domain $U$ whose relations on $n$-tuples are determined by the orbits of the action of $\mathrm{G}$ on the set of  $n$-tuples, is homogeneous. See  \cite{cameronbook}; or \cite{Fra} for a more extensive exposition. In particular it follows along the discussion in Section 2.4 of \cite{cameronbook} that $\overG\cap {\mathfrak S}(U)$ is a Polish group, see \cite{Kechris},   with respect to the function topology and composition as group operation.

A  relational structure $\mathrm{R}$ with domain $U$  is \emph{prehomogeneous},  (see \cite{pabion}),  if for every finite set $F\subseteq U$  there is a finite set $F'\supseteq F$   such that every  isomorphism $f$ of $\restrict{\mathrm{R}}{F}$ to  $\restrict{\mathrm{R}}{f[F]}$ does then   extend  to an automorphism,  provided  that it extends to $F'$. If the cardinality  of $F'$ can be  bounded by some function $\theta$ depending on the cardinality of $F$ then $\mathrm{R}$ is \emph{uniformly prehomogeneous}. It is the  case that if a relational structure $\mathrm{R}$ is homogeneous or prehomogeneous then $\emb(\mathrm{R})= \overline{\aut(\mathrm{R})}$ and the set of copies of $\mathrm{R}$ is equal to the set of copies $\overline{\aut(\mathrm{R})}[U]$ for the group $\aut(\mathrm{R})$. 
Actually,  if $\aut(\mathrm{R})$ is oligomorphic and $\emb(\mathrm{R})= \overline{\aut(\mathrm{R})}$ then $\mathrm{R}$ is prehomogeneous, see~\cite{siblings}  Theorem~2.2. A  group $\mathrm{G}\subseteq \mathfrak{S}(U)$  being \emph{oligomorphic} if for every  $n$ the number of orbits of $n$-element subsets of $U$ is finite. This amounts to the fact, see \cite{ryll},  that the group $\overline{\mathrm{G}}\cap {\mathfrak S}(U)$ is the automorphism group of an $\aleph_0$-categorical theory.

\section{An outline of the results}

Results on the  theme of copies of an homogeneous structure  have been developped in a series of  papers by Kurili\'c, Kuzeljevi\'c and Todor\u{c}evi\'c (e.g. \cite{kurilic1, kurilic2, kurilic-todorcevic, kurilic-kuzeljevic}).  Some of the results presented here have been the subject of a lecture  by the third author in March 2019 at the FG1 Seminar of the Technical University of Vienna \cite{sauer}.

For this section let  $U$ be a countable  set and let $\mathrm{G}$ be a subgroup of the symmetric group $\mathfrak{S}(U)$.

\begin{thm} \label{thm:oligomorphic}[Theorem \ref{thm:Woligomorphic2}]
If the group $\mathrm{G}$ is oligomorphic then  there is an embedding of $(\powerset (\omega);\subseteq)$, the power set of $\omega$ ordered by  inclusion,  into  $(\overline{\mathrm{G}}[U];\subseteq)$, the set of copies for the group $\mathrm{G}$,  ordered by inclusion.
\end{thm}

For $F$ a finite subset of $U$ let $\mathrm{G}\langle F\rangle=\{g\in \mathrm{G}\mid \forall x\in F\,  (g(x)=x)\}$ and let  $\mathfrak{ac}(F)$, the {\em algebraic closure} of $F$,  be the union of all finite point orbits of the stabilizer group $\mathrm{G}\langle F\rangle$. If $S\subseteq U$ is infinite let $\mathfrak{ac}(S):=\bigcup \{\mathfrak{ac}(F)\mid F\in \powerset_{<\omega}(S)\}$.  Note that  if $\mathrm{G}$ is oligomorphic then $\mathfrak{ac}(F)$ is finite for every finite subset $F$ of $U$. A group $G$ of permutations of $U$ is \emph{algebraically finite} if the algebraic closure of every finite subset of $U$ is finite. Oligomorphic groups are algebraically finite. There are many others.

\begin{thm}\label{thm:twocop}[Theorem \ref{thm:constrpairextbas}]  
If $\mathrm{G}$ is algebraically finite then for every finite subset $F$ of $U$ the algebraic closure of $F$,  $\mathfrak{ac}(F)$, is the intersection of two copies for the group  $\mathrm{G}$.  
\end{thm}

\begin{problem} 
Is it true that $\mathfrak{ac}(S)$ is  an intersection of the set of  copies containing $S$  if $S$ is infinite and the group $\mathrm{G}$ is oligomorphic?
\end{problem}

\begin{problem} 
If $\mathrm{G}$ is algebraically finite, is it true that  $(\overline{\mathrm{G}}[U];\subseteq)$  embeds $(\powerset (\omega); \subseteq)$?
\end{problem}

\begin{thm}\label{thm:basictup}[Theorem \ref{thm:characterization}]  A subset $A$ of $U$ is a copy, that is belongs to $\overG[U]$,  if and only if for every finite subset $F$ of $A$  and every $x\in U\setminus F$ there is a function $g\in \mathrm{G}\langle F\rangle$ with $g(x)\in A$. 
\end{thm}

Using  Theorem \ref{thm:basictup} we obtain: 
\begin{cor}\label{cor:basictyp}[Lemma \ref{lem:updirected},    Lemma \ref{lem:Gsubdelta}, Corollary \ref{lem:isolated}]
The set of copies $\overline{\mathrm{G}}[U]$ is closed under the union of non-empty up-directed families.  Topologically, the set of copies  $\overline{\mathrm{G}}[U]$ is  a $\mathrm{G}_{\delta}$-set of $\powerset (U)$ equipped with the powerset topology. (See Definition \ref{defin:contpowerbrr}.) Moreover, if $\overline{\mathrm{G}}[U]$ contains more than one copy, then it has no isolated point with respect to the powerset topology.   
\end{cor}

Corollary  \ref{cor:basictyp} will then be used to establish that:

\begin{thm}\label{thm:introcard}[Theorem \ref{them:cardcopies}] The cardinality of $\overline{\mathrm{G}}[U]$ is either $1$ or $2^{\aleph_0}$. 
\end{thm}  

\begin{problem} Is it true that if $\overline{\mathrm{G}}[U]$ has more than  one element then the partial ordered set  $(\overline{\mathrm{G}}[U];\subseteq)$  is not totally ordered by inclusion? Is it true that it embeds $(\powerset (\omega);\subseteq)$?
\end{problem}

In  Section \ref{sect:rankedcl} we will    introduce the notions of being ranked or unranked for point orbits of the stabilizer subgroups $\mathrm{G}\langle F\rangle$ for finite subsets $F$ of $U$. Assigning an ordinal number rank to the ranked point orbits of those stabilizer subgroups. Leading  to the notion  of  {\em ranked closure} $\mathfrak{rc}(S)$  of $S$, for  subsets $S$ of $U$.  It turns out, Lemma~\ref{lem:locfinacrc},  that if the group $\mathrm{G}$ is algebraically finite then $\mathfrak{rc}(S)=\mathfrak{ac}(S)$  for every  subset $S$ of $U$.  

\begin{thm}\label{thm:outlnem}[Theorem \ref{thm:rcclinters}]  The ranked closure of a finite set $F$ is equal to  the intersection of all copies containing $F$. There is just one copy, namely $U$,  for a group $\mathrm{G}\subseteq \mathfrak{S}(U)$, if and only if   the ranked closure of the empty set is $U$ if and only if every point orbit of the group $\mathrm{G}$ is ranked. 
\end{thm} 

We do not know if a more general version of  Theorem \ref{thm:twocop} holds for all  groups $\mathrm{G}$ with $\overG[U]\not=\{U\}$. The next Theorem,  being just of an enough strengthening of Theorem \ref{thm:outlnem} in order to establish  Theorem~\ref{thm:twocop},  does not seem to be sufficiently strong to lead to a generalization of Theorem \ref{thm:twocop}. Note that Theorem \ref{thm:twocop} implies that $\overG[U]\not=\{U\}$ if the group $\mathrm{G}$ is algebraically finite. 

 \begin{thm}\label{thm:interschcopi}[Theorem \ref{thm:interschcop}] 
  Let $U$ be a countable set and $\mathrm{G}$ be a subgroup of $\mathfrak S(U)$ and $\overG[U]\not=\{U\}$.  Then for every finite subset $F$ of $U$ and every copy $C\in \overG[U]$ with $F\subseteq C$ (hence $\mathfrak{rc}(F)\subseteq C$),  there exists: \\
  An $\omega$-sequence  
  $C=C_0\supset C_1\supset C_2\supset C_3\supset\dots$ of copies $C_i\in \overG[U]$ for which $\mathfrak{rc}(F)=\bigcap_{i\in \omega}C_i$.   
 \end{thm}

For the connection between  the Hausdorff rank of a linear order and the rank defined in this article, see   Subsection \ref{subs:linordcas}.  

Besides the set of copies $\overG[U]$, there are three related  sets of interest:   the set of intersections of copies  $\widehat{\overG[U]}$, the set $\overG^{\mathfrak{rc}}[U]$of ranked closures of arbitrary subsets of $U$ and the  topological closure $\overline{\overG[U]}$ of ${\overG[U]}$ in  $\powerset(U)$.

\begin{lem}\label{lem:equalityclosures}
$\overG[U]\subseteq \widehat{\overG[U]} \subseteq \overG^{\mathfrak{rc}}[U]=\overline{\overG[U]}$. 
\end{lem}
The content of this lemma is in the equality $\overG^{\mathfrak{rc}}[U]=\overline{\overG[U]}$. Indeed, 
the first inclusion is obvious.  The family $\overG^{\mathfrak{rc}}[U]$ defines an algebraic closure system. As a closure system,  it is closed under intersection, since it contains $\overG[U]$ it contains  $\widehat{\overG[U]}$. Being algebraic, it is closed in $\powerset(U)$ equipped with the product topology. Since it contains ${\overG[U]}$, it contains its topological closure, that is $\overline{\overG[U]}$. It remains to show that $\overG^{\mathfrak{rc}}[U]\subseteq\overline{\overG[U]}$.  Note that a subset $S$ of $U$ is in $\overline{\overG[U]}$  if for every finite subset $F\subseteq S$ and every finite subset $E$ of $U\setminus S$ there exists a copy $C\in \overG[U]$ for which $F\subseteq C$ and $E\cap C=\emptyset$. Thus, let $S\in \overG^{\mathfrak{rc}}[U]$.  Let $F\subseteq S$ be a finite subset,  and  let $E$ be a finite subset  of $U\setminus S$. Since $F$ is finite, $\mathfrak{rc}(F)$ is equal to $\mathfrak{ic}(F)$, the intersection of copies containing $F$ (Theorem \ref{thm:outlnem}), hence for every $x\not \in F$ there is a copy $C$ containing $F$ and not $x$. In fact, there is some copy  $C$  containing  $F$ and  no element of $E$  (see Lemma \ref{cor:in/out}). Hence, $\mathfrak{ic}(F)\in \overline{\overG[U]}$.   

The fact that $\overline{\overG[U]}=\overG^{\mathfrak{rc}}[U]$  has the following significant consequence: 

\begin{cor} $\overline{\overG[U]}$, the topological closure of the set of copies, is closed under intersections. 
\end{cor} 

As  an other consequence, we have:
 \begin{thm}\label{thm:equrcpwscl}[Theorem \ref{thm:alrcinterscl}]
Let $U$ be a countable set and $\mathrm{G}$ be a subgroup of $\mathfrak S(U)$. Then $\mathfrak{rc}(S)$ is the least element of $\overline{\overG[U]}$ for every subset $S$ of $U$. 
\end{thm}

The intersection of the copies containing a finite set is equal to the ranked closure of this finite set. We do not know if this is also the case for infinite subsets of $U$. That is we cannot decide:

\begin{problem}
Is it true that $\widehat{\overG[U]}=\overline{\overG[U]}$? 
\end{problem}

But,  Lemma \ref{lem:alrcinterscl} states   that $\widehat{\overG[U]}$ is an algebraic closure system  if and only if $\widehat{\overG[U]}=\overline{\overG[U]}$.

 \section{Elementary properties}

Let  $U$ be a countable  set and let $\mathrm{G}$ be a subgroup of the symmetric group $\mathfrak{S}(U)$. 
 
\begin{defin}\label{defin:contpowerbrr}
For $\mathscr{Q}\subseteq \powerset(U)$ and $S\cup T\subseteq U$ let $\mathscr{Q}\llbracket S\rrg=\{V\in \mathscr{Q}\mid S\subseteq V\}$ and let $\mathscr{Q}\llbracket S,T\rrbracket=\{V\in \mathscr{Q}\mid \text{ $S\subseteq V$ and $T\cap V=\emptyset$}\}$.

For $E$ and $F$ two finite subsets of $U$  the set $\powerset(U)\llbracket F,E\rrbracket$ is a basic neighbourhood for the {\em powerset topology}.
\end{defin}

Of course $\overG[U]\subseteq \powerset(U)$. Then  $\overline{\mathrm{G}}[U]\llg F\rrg$, for $F\subseteq U$, denotes the set of copies for the group $\mathrm{G}$ which contain $F$ as a subset.  The set $\overG[U]\llbracket F,E\rrbracket$ is then a neighbourhood for the powerset topology relative to the set of copies $\overG[U]$.  We investigate the set of copies $\overG[U]$ for the group $\mathrm{G}$ with respect to the order $\subseteq$ on $\powerset(U)$ and with respect to the  powerset topology on $\powerset(U)$.   Then $\overline{\overG[U]}$ is the set of subsets of $U$ which are in the topological closure of $\overG[U]$ the set of copies  for the group $\mathrm{G}$ and $(\overG[U];\subseteq)$ is the {\em poset of copies} and $\widehat {\overline{\mathrm{G}}[U]}$ is the set of intersections of copies.   The \emph{hypergraph of copies} is the hypergraph  $H(\overline G):=(U, \overline{\mathrm{G}}[U])$ whose hyperedges are the copies.

\begin{defin}\label{defin:stabilzer}
  Let $X$ be a subset of $U^U$ and $F$ be  a subset of $U$. Then $X\langle F\rangle:=\{f\in X\mid \restrict{f}{F}=\restrict{\id\, }{\, F}\}$ denotes the {\em stabilizer} of $F$  for the set $X$ of  functions and $X_{ F} =\{f\in X\mid f[F]=F\}$ denotes the {\em setwise stabilizer} of $F$  for the set $X$ of  functions. Also,  set $\overline X$ the topological closure of $X$ in $U^U$ equipped with the product topology. 
  \end{defin} 

\begin{lem}\label{lem:extov}
Let $X\subseteq U^U$ and $F\in \finp{U}$. Then 
\begin{enumerate}
\item $\overline X\langle F\rangle = \overline {X\langle F\rangle }$.
\item $\overline X_{ F}= \overline {X_{ F}}$.
\end{enumerate}
\end{lem}
\begin{proof}
The sets of functions $\overline X\langle F\rangle$ and $\overline{X}_{ F}$ are closed because $\overline{X}$ is closed and $F$ is finite. Hence $\overline {X\langle F\rangle}\subseteq \overline X\langle F\rangle$ and  $ \overline {X_{ F}}\subseteq \overline X_{ F}$. Conversely, let $f\in  \overline X\langle F\rangle$. Let $H\in \finp{U}$.  Since $f\in \overline X$ there is some $g\in X$ such that $\restrict{g}{(F\cup H)}=\restrict{f}{(F\cup H)}$.  Since  $g\in X\langle F\rangle$ and $f$ and   $g$ coincide on $H$ this implies $f\in \overline {X\langle F\rangle}$. 

Let $f\in  \overline X_{ F}$ and let $H\in \finp{U}$.  Since $f\in \overline X$ there is some $g\in X$ such that $\restrict{g}{(F\cup H)}=\restrict{f}{(F\cup H)}$.  Since  $g\in X_{F}$ and $f$ and   $g$ coincide on $H$ this implies $f\in \overline {X_{ F}}$. 
 \end{proof}

\begin{lem}\label{lem:fix-copy}  If $F$ is a finite subset of a copy $C\in \overG[U]$ then there exists a function $f\in \overG$ with $f[U]=C$ and $f(a)=a$ for all $a\in F$. That is: $\overline{\mathrm{G}\langle F\rangle}[U]=\overG\langle F\rangle[U]=  \overG[U]\llg F\rrg$ for every $F\in \finp{U}$. 
 \end{lem}
 \begin{proof} The following inclusion for every subset $F$ of $U$ (finite or not) is obvious: $\overG\langle F\rangle[U]\subseteq   \overG[U]\llg F\rrg$. 
We claim that the  reverse inclusion holds for every $F\in \finp{U}$.  Let $C\in \overG[U]\llg F\rrg$. Let $g\in \overline G$ such that $g[U]=C$. Let $F'$ be  such that $g(F')=F$. Then there is an  $h\in G$ such that $\restrict{h}{F'}=\restrict{g}{F'}$.  Set $f:= g\circ h^{-1}$. Then $C=f[U]$ and  $f\in \overG\langle F\rangle $ implying $C\in (\overG\langle F\rangle)[U]$. The claim follows. The equality follows too. 

\end{proof}

\begin{cor}\label{cor:fix-copy}
$(\overG\langle F\rangle)[U]= \overG_{ F}[U]= \overG[U]\llg F\rrg$ for every $F\in \finp{U}$.
\end{cor}

\begin{cor}\label{cor:in/out}
If $F$ is a finite subset of a copy $C\in \overG[U]$ then $\mathfrak{ic}(F)$, the intersection of copies containing $F$ is equal to $\mathfrak{ic}_C(F)$ the intersection of copies containing $F$ and included into $C$. Furthermore, if $E$ is a finite set disjoint from $\mathfrak{ic}(F)$, there is a copy containing $F$ and disjoint from $E$. \end{cor}
\begin{proof} Clearly, $\mathfrak{ic}(F)\subseteq \mathfrak{ic}_C(F)\subseteq C$. Let $x\in C \setminus \mathfrak{ic}(F)$. Let $C'\in \overG[U]$. Set $F':= F\cup \{x\}$. Apply   Lemma \ref {lem:fix-copy} to 
$C$ and $F'$. There is some $f\in \overline G\langle F'\rangle$ such that $f[U]=C$. Then $C'':= f[C']$ is a copy containing  $F$ included in $C$ and not containing $x$, hence $x\not \in\mathfrak{ic}(F)$. Now, we prove that  there is some copy $D$ containing $F$ and disjoint from $E$. Indeed, pick $x\in E$, then some copy containing $F$, say  $D$,  avoids $x$.   The property above ensures that $\mathfrak{ic}(F)$, the intersection of copies containing $F$ is equal to  $\mathfrak{ic}_D(F)$, the intersection of copies containing $F$ that are included into $D$. If $D$ contains some  element $y$ of $E$ we may find a copy $D'$ included into $D$ and  avoiding it. Repeating the process, we find a copy containing no element of $E$. 
\end{proof}
  
\begin{lem}\label{lem:restriction} 
Let $\mathrm{G}$ be a subgroup of $\mathfrak S(U)$ and let  $\{f,g\}\subseteq  \overG$ with $f[U]\subseteq g[U]$. Then $g^{-1}\circ f\in \overG$. If $f[U]= g[U]$ then $g^{-1}\circ f\in \overG\cap \mathfrak{S}(U)$.
 \end{lem}
 \begin{proof} 
Let $F\in\powerset_{<\omega}(U)$. There exists a function $k\in\mathrm{G}$ with $\restrict{k}{F}=\restrict{f}{F}$ and there exists a function $h\in \mathrm{G}$ with $\restrict{h}{(g^{-1}[k[F]])}=\restrict{g}{(g^{-1}[k[F]])}$. Implying that $\restrict{(h^{-1}\circ k)}{F}=\restrict{(g^{-1}\circ f)}{F}$ with $h^{-1}\circ k\in \mathrm{G}$.   Let $f[U]= g[U]$. Then $g^{-1}\circ f$ maps $U$ onto $U$. 
\end{proof}

 \begin{lem}\label{lem:justonecopy}
 Let ${\mathrm{G}}$ be a subgroup of $\mathfrak S(U)$. Then the following properties are equivalent:
 \begin{enumerate} [(i)]
\item $\overG[U]= \{U\}$;
 \item  $\overG\subseteq \mathfrak S(U)$;
 \item $\overline {\mathrm{G}\langle F\rangle}\subseteq\mathfrak S(U)$ for every $F\in \finp{U}$;
  \item  $\overline {\mathrm{G}\langle F\rangle}\subseteq\mathfrak S(U)$ for some  $F\in \finp{U}$.
 \end{enumerate}
 \end{lem}
 \begin{proof}
 The equivalence between $(i)$ and $(ii)$ and the implications $(ii)\Rightarrow (iii)\Rightarrow (iv)$ are obvious. Suppose that $(iv)$ holds. We will  prove that $(ii)$ holds. Let $f\in \overG$. Let $F\in \finp{U}$ such that  $\overline {\mathrm{G}\langle F\rangle}\subseteq\mathfrak S(U)$  and let $g \in \mathrm{G}$ such that $\restrict{g}{F}= \restrict{f}{F}$. Let $k:= g^{-1}\circ f$. Since $\{g^{-1}, f\}\subseteq  \overG$ the function  $k\in\overG$, actually $k\in \overG\langle F\rangle$. Implying  $k\in \overline {\mathrm{G}\langle F\rangle}$ according to Lemma \ref{lem:extov}.  Hence $k\in\mathfrak S(U)$. Consequently $f= g\circ k\in \mathfrak S(U)$. 

\end{proof} 
 
 It follows that if the stabilizer $\mathrm{G}\langle F\rangle$ of $\mathrm{G}$ for some finite subset $F$ of $U$ contains only the identity function then the group $\mathrm{G}$ is closed in $\mathfrak S(U)$. 

 \begin{lem}\label{lem:onecopy}
The copy $U$ is isolated, for the powerset topology relative to the set of copies $\overline{\mathrm{G}}[U]$   if and only if $\overline{\mathrm{G}}[U]=\{U\}$. 
\end{lem}
\begin{proof}
The set  $U$ is isolated if and only if  $\overG[U]\llg F\rrg = \{U\}$ for some $F\in \finp{U}$. According to Lemma \ref{lem:fix-copy} and Lemma \ref{lem:extov} $ \{U\}=\overline{\mathrm{G}}[U]\llg F\rrg$ if and only if   $\{U\}=(\overG\langle F\rangle)[U]=(\overline{\mathrm{G}\langle F\rangle})[U]$ if and only if $\mathfrak{S}(U)\supseteq \overline{\mathrm{G}\langle F\rangle}$ if and only if $\overline{\mathrm{G}}[U]=\{U\}$ according to Lemma \ref{lem:justonecopy}.  

\end{proof}

\begin{lem}\label {lem:stronglyisolated} 
For every copy $V\in \overline{\mathrm{G}}[U]\setminus \{U\}$ and subset $F\in \finp{V}$ there exists a copy $W\in \overline{\mathrm{G}}[U]$  with $F\subseteq W\subset V$.  
\end{lem}
\begin{proof}
 Let $V\in \overline{\mathrm{G}}[U]\setminus \{U\}$ and $F\in \finp{V}$. 
  According to Lemma~\ref{lem:fix-copy} there exists a function  $g\in \overG\langle F\rangle$ with  $g[U]= V$. Set $W:= g[V]$. 
\end{proof}

\begin{cor}\label{lem:isolated}
The set of copies $\overline{\mathrm{G}}[U]$ has no isolated point for the powerset topology if and only if $\overline{\mathrm{G}}[U]\not =\{U\}$. 
\end{cor}
\begin{proof}
Suppose that a copy $V$ is isolated in $\overline{\mathrm{G}}[U]$.   Let $f\in \overG$ be a function with $f[U]=V$.   The set of copies $\overline{\mathrm{G}}[U]\llg F, E \rrg$, containing $F$ and being disjoint from $E$,  is open for every $\{F,E\}\subseteq \finp{U}$.  Because $V$ is isolated,  we may select $F,E$ so that $\overline{\mathrm{G}}[U]\llg F, E\rrg= \{V\}$. Let  $F':= f^{-1}(F)$. Then $\overline{\mathrm{G}}[U]\llg F'\rrg= \{U\}$. For otherwise there exists a copy $W$ with $F'\subseteq W\subset U$. Then $F\subseteq f[W]\subset V$ with $f[W]\cap E\subseteq V\cap E=\emptyset$; in contradiction to $\overline{\mathrm{G}}[U]\llg F, E\rrg= \{V\}$. The equality $\overline{\mathrm{G}}[U]\llg F'\rrg= \{U\}$ implies, according to Lemma \ref {lem:onecopy},  that $\overline{\mathrm{G}}[U]=\{U\}$. 

\end{proof}

\section{Orbits and Types and Typesets}\label{sect:Torb}

Let $U$ be a set and $\mathrm{G}$ be a subgroup of $\mathfrak S(U)$. 

In order to indicate an orbit of the action of the group $\mathrm{G}$ it suffices to choose any element in the orbit. Let $\mathrm{Orb}(\mathrm{G})$ denote the set of different orbits of the action of $\mathrm{G}$ on $U$.  The orbits of the action of the stabilizer subgroups of the group $\mathrm{G}$ will play a role and we will need some particular notation for them. According to Definition \ref{defin:stabilzer} for $F\subseteq U$ the  $F$-stabilizer of $\mathrm{G}$ will be  denoted by $\mathrm{G}\langle F\rangle$, that is:
\[
\mathrm{G}\langle F\rangle=\{g\in \mathrm{G}\mid \forall\, x\in F\, (g(x)=x)\}. 
\]
Then the set of orbits, denoted $\mathrm{Orb}(\mathrm{G}\langle F\rangle)$,  of the subgroup $\mathrm{G}\langle F\rangle$ of $\mathrm{G}$  acting on $U$ forms a partition of the set $U$.  We emphasize the fact that the orbit  of any element $x$ of $F$  w.r.t. the action of  $\mathrm {G}\langle F\rangle$ is $\{x\}$.  A   {\em type} is a pair of the form $\langle F\tr p\rangle$ for which  $F$ is a finite subset of $U$ and $p\in U$.   The set $F$ is the {\em sockel} of the type $\langle F\tr p\rangle$.   Two types $\langle F\tr p\rangle$ and $\langle E\tr q\rangle$ are {\em equal} if $F=E$ and if there exists a function $g\in \mathrm{G}\langle F\rangle$ with $g(p)=q$. The {\em typeset}, denoted $\mathrm{G}\langle F\tr p\rangle$,  of the type $\langle F\tr p\rangle$, is the set in $\mathrm{Orb}(\mathrm{G}\langle F\rangle)$ given by:  
\[
\mathrm{G}\langle F\tr p\rangle:=\{x\in U\mid \exists\, g\in \mathrm{G}\langle F\rangle\, (g(p)=x)\}.  
\]
Note that for all $\{x,y\}\subseteq \mathrm{G}\langle F\tr p\rangle$ there exists a function $f\in \mathrm{G}\langle F\rangle$ with $f(x)=y$. It follows that two types $\langle F\tr p\rangle$ and $\langle E\tr q\rangle$ are {\em equal} if and only if  $F=E$ and if  $\mathrm{G}\langle F\tr p\rangle=\mathrm{G}\langle E\tr q\rangle$. That is, the set of typesets of the types with sockel $F$ is the set $\mathrm{Orb}(\mathrm{G}\langle F\rangle)$.   Two typesets having the same sockel are either disjoint or equal. Note:

\begin{fact}\label{fact:extendtypcl}
Let $\langle F\tr p\rangle$ be a type. Then 
\[
\mathrm{G}\langle F\tr p\rangle=\{x\in U\mid \exists\, f\in \overG\langle F\rangle\, (f(p)=x)\}.  
\]
\end{fact}

The type $\langle E\tr q\rangle$ is a {\em continuation} of the type $\langle F\tr p\rangle$ if $F\subseteq E$ and if $q\in \mathrm{G}\langle F\tr p\rangle$. It follows that the type $\langle E\tr q\rangle$ is a  continuation of the type $\langle F\tr p\rangle$ if and only if $F\subseteq E$ and $\mathrm{G}\langle F\tr p\rangle\supseteq \mathrm{G}\langle E\tr q\rangle$. Note:

\begin{fact}\label{fact:partitiont}
Let $\mathrm{G}\subseteq \mathfrak{S}(U)$. Let $\langle F\tr p\rangle$ be a type  and let $ E\in \finp{U}$ with $F\subseteq E$. Then the   typesets of the types  $\langle E\tr q\rangle$ which are continuations of the type $\langle F\tr p\rangle$ form a partition of the typeset $\mathrm{G}\langle F\tr p\rangle$. That is:
\[
\mathrm{G}\langle F\tr p\rangle=\bigcup_{q\in \mathrm{G}\langle F\tr p\rangle}\mathrm{G}\langle E\tr q\rangle. 
\]
Noting that  either $\mathrm{G}\langle E\tr q\rangle=\mathrm{G}\langle E\tr q'\rangle$ or $\mathrm{G}\langle E\tr q\rangle\cap\mathrm{G}\langle E\tr q'\rangle=\emptyset$ and that  a type $\langle E\tr q\rangle$ with $F\subseteq E$ is a continuation of the type $\langle F\tr p\rangle$ if and only if $q\in \mathrm{G}\langle F\tr p\rangle$.  
\end{fact}

For $1\leq n\in \omega$: Let  $U^n$ denote the set of $n$-tuples with entries in $U$. Let $\mathrm{G}_n$ be the subgroup of $\mathfrak{S}(U^n)$  induced by the  action of the group $\mathrm{G}$ on the set $U^n$. Let $\powerset_n(U)$ denote the set of $n$-element subsets of $U$. Let $\mathrm{G}_{(\powerset_n)}$ be the subgroup of  $\mathfrak{S}(\powerset_n(U))$ induced by the  action of the group $\mathrm{G}$ on the set $\powerset_n(U)$. 

\begin{lem}\label{lem:testoligo} 
The following properties are equivalent:
\begin{enumerate}
\item The set  $\mathrm{Orb}(\mathrm{G}\langle F\rangle)$ is finite for every $F\in \finp{U}$. 
\item The set   $\mathrm{Orb}(\mathrm{G}_n)$ is finite for every  $1\leq n\in \omega$. 
\end{enumerate} 
\end{lem}
\begin{proof}
Note: $\mathrm{Orb}(\mathrm{G}\langle \emptyset\rangle)=\mathrm{Orb}(\mathrm{G})=\mathrm{Orb}(\mathrm{G}_1)$. The equivalence of conditions (1) and (2) follows then via induction on $n$ and the fact that two $n+1$ tuples $\vec{x}=(x_0,x_1,\dots,x_{n-1},x_n)\in U^{n+1}$ and $\vec{y}=(y_0,y_1,\dots,y_{n-1},y_n)\in U^{n+1}$ are in the same orbit of $\mathrm{G}_{n+1}$ if and only if there exists an $n$-tuple $(u_0,u_1,\dots,u_{n-1})\in U^n$ with functions $g$ and $f$ in $\mathrm{G}$ so that $g(x_i)=u_i=f(y_i)$ for all $i\in n+1$ and so that the typesets
$\mathrm{G}\langle \{u_0,u_1,\dots,u_{n-1}\}\tr g(x_n)\rangle$ and $\mathrm{G}\langle \{u_0,u_1,\dots,u_{n-1}\}\tr f(y_n)\rangle$ are equal. 

\end{proof}

\begin{defin}\label{defin:oligomorphic3} 
If the second  condition of Lemma \ref{lem:testoligo}  is satisfied the group is said to be \emph{oligomorphic}. 
\end{defin}

\begin{lem}\label{lem:maptypef}
Let $\langle F\tr x\rangle$ be a type. If $f\in \overG$ then $\mathrm{G}\langle f[F]\tr f(x)\rangle\supseteq f[\mathrm{G}\langle F\tr x\rangle]$. If in addition the function $f$ is onto $U$, that is  $f\in \overG\cap \mathfrak{S}(U)$, then 
 $\mathrm{G}\langle f[F]\tr f(x)\rangle=f[\mathrm{G}\langle F\tr x\rangle]$. 
\end{lem}
\begin{proof}
Let $y\in \mathrm{G}\langle F\tr x\rangle$. There exists a function $g\in \mathrm{G}\langle F\rangle$ with $g(x)=y$.  Then $(f\circ g\circ f^{-1})(f(x))=f(y)$ and $(f\circ g\circ f^{-1})(f(z))=f(z)$ for all $z\in F$.   Hence $f(y)\in \mathrm{G}\langle f[F]\tr f(x)\rangle$ and $\mathrm{G}\langle f[F]\tr f(x)\rangle\supseteq f[\mathrm{G}\langle F\tr x\rangle]$.

Let  $f\in \overG\cap \mathfrak{S}(U)$.   Let $y\in \mathrm{G}\langle f[F]\tr f(x)\rangle$. There exists a function $h\in \mathrm{G}\langle f[F]\rangle$ with $y=h(f(x))$. Then $(f^{-1}\circ h\circ f)(z)=z$ for all $z\in F$, and $f(f^{-1}\circ h \circ f)(x)= f \circ f^{-1}\circ h \circ f(x)= h \circ f(x)= y$. Hence $y\in  f[\mathrm{G}\langle F\tr x\rangle]$ and $\mathrm{G}\langle f[F]\tr f(x)\rangle\subseteq f[\mathrm{G}\langle F\tr x\rangle]$. 
\end{proof}

\begin{cor}\label{cor:maptypef}
Let  $\langle F\tr x\rangle$ be a type whose typeset $\mathrm{G}\langle F\tr x\rangle$ is finite. Then $\mathrm{G}\langle f[F]\tr f(x)\rangle=f[\mathrm{G}\langle F\tr x\rangle]$ for all functions $f\in \overG$.   The set 
$\mathrm{G}\langle F\tr x\rangle$ is a subset of all copies $C\in \overG[U]$ with $F\subseteq C$. 
\end{cor}
\begin{proof}
The first assertion of the Corollary follows from Lemma \ref{lem:maptypef} because
there exists a function $h\in \mathrm{G}$ for which $\restrict{h}{(F\cup \mathrm{G}\langle F\tr x\rangle)}=\restrict{f}{(F\cup \mathrm{G}\langle F\tr x\rangle)}$. Let $C\in \overG[U]$ with $F\subseteq C$. It follows from Lemma~\ref{lem:fix-copy}  that  there exists a function $f\in \overG\langle F\rangle$, hence with $f[F]=F$, for which $f[U]=C$. Then $\mathrm{G}\langle F\tr x\rangle=\mathrm{G}\langle F\tr f(x)\rangle=f[\mathrm{G}\langle F\tr x\rangle]\subseteq C$. The two types $\langle F\tr x\rangle$ and $\langle F\tr f(x)\rangle$ are equal because there exist a function $g\in \mathrm{G}$ with $\restrict{g}{(F\cup \{x\})}=\restrict{f}{(F\cup \{x\})}$. That is $g\in \mathrm{G}\langle F\rangle$ with $h(x)=f(x)$. 
\end{proof}

\section{A characterization of members of $\overline{\mathrm{G}}[U]$}

\begin{thm}\label{thm:characterization} 
Let $U$ be a countable set and $\mathrm{G}$ be a subgroup of $\mathfrak S(U)$. A subset $C$ of $U$ belongs to $\overline{\mathrm{G}}[U]$ if and only if  for every $F\in \finp{C}$ and every $x\in U\setminus F$ there is a function $g \in  \mathrm{G}\langle F\rangle$ such that  $g(x)\in C$. 
\end{thm}
\begin{proof}
Let $(*)$ be the following statement:
\noindent
$(*)$: For every $F\in \finp{C}$ and every $x\in U\setminus F$ there is a function $g \in  \mathrm{G}\langle F\rangle$ such that  $g(x)\in C$. 

Let $C\in \overline{\mathrm{G}}[U]$ and let $F\in \finp{C}$. Let $x\in U\setminus F$.   According to Lemma \ref{lem:fix-copy} there exists a function $h\in \overG\langle F\rangle$ with $h[U]=C$. Hence $h(x)\in C$.  Since $h\in \overG\langle F\rangle$ there exists a function  $g\in  \mathrm{G}\langle F\rangle$ such that $\restrict{h}{(F\cup \{x\})}=\restrict{g}{(F\cup\{x\})}$.  Hence $g(x)\in C$. We conclude that property $(*)$  holds.

Given property $(*)$ we make use of the standard back and forth method in order to prove that there is a function $f\in \overG$  with $f[U]=C$. (See  Fra\"{\i}ss\'e's characterization of homogeneous structures, \cite{Fra}.) Let $\{u_i\mid i\in \omega\}$ be an $\omega$-enumeration of $U$ and let $U_n=\{u_i\mid i\in n\}$.  Let $\{c_i\mid i\in \omega\}$ be an $\omega$-enumeration of $C$ and let $C_n=\{c_i\mid i\in n\}$. For $n\in \omega$ we will construct recursively sets   $V_n\subseteq U$ and functions $l_n\in \mathrm{G}$ with $l_n[V_n]\subseteq C$ and with $f_n:=\restrict{l_n}{V_n}\subseteq \restrict{l_{n+1}}{V_{n+1}}:=f_{n+1}$. Furthermore $U_n\subseteq V_n$ and $C_n\subseteq l_n[V_n]$   for every $n\in \omega$. Then $f:=\bigcup_nf_n\in \overG$ with $f[U]=U$. 

Let $V_0=\emptyset$  and $l_0=\emptyset$. Given  the set $V_n$ and the function  $l_n\in \mathrm{G}$   let $i\in \omega$ be the smallest number for which $u_i\not\in V_n$. Let $j\in \omega$ be the smallest number for which $c_j\not\in l_n[V_n]$. Using property $(*)$ let $g\in \mathrm{G}\langle l_n[V_n]\rangle$ with $g(l_n(u_i))\in C$. Let $V_{n+1}=V_n\cup \{u_i, l^{-1}(g^{-1}(c_j))$ and let $l_{n+1}=g\circ l_n$.  

\end{proof}

\begin{cor}\label{cor:characterization} 
Let $U$ be a countable set and $\mathrm{G}$ be a subgroup of $\mathfrak S(U)$. A subset $C$ of $U$ belongs to $\overline{\mathrm{G}}[U]$ if and only if  $C\cap \mathrm{G}\langle F\tr x\rangle\not=\emptyset$ for every type $\langle F\tr x\rangle$ of $\mathrm{G}$ with $F\subseteq C$.  
\end{cor}

 \begin{lem} \label{lem:bernstein}
 Let $U$ be a countable set and $\mathrm{G}$ be a subgroup of $\mathfrak S(U)$.  If for every $F\in  \finp{U}$ every orbit of the action of $G\langle F\rangle$ on $U\setminus F$ is infinite then there is a coinfinite subset $A$ of $U$ such that $A'\in \overline{\mathrm{G}}[U]$ for every subset $A'$ of $U$ which contains $A$. In particular, there is an embedding of $\powerset (\omega)$ ordered by inclusion into $\overline{\mathrm{G}}[U]$ ordered by inclusion. 
\end{lem}
\begin{proof} Let $Type(G):= \{Orb(x, G\langle F\rangle) : F\in  [U\setminus \{x\}]^{<\omega}\;  \text{and}\;  x\in U\}$. This is a countable set of subsets of $U$. According to our hypothesis, each one is infinite. Bernstein's property ensures that there is a partition of $U$ into two infinite parts $A$ and $B$  such that each member of $Type (G)$ meets $A$ and $B$ (for this general property of sets, replace the continuum in  Lemma 2, p.514, of \cite {kuratowski} by $\aleph_0$). If $A'$ is any subset of $U$ containing $A$ then each member of $Type (G)$ will meet  $A' $. In particular,  the condition of Corollary \ref{cor:characterization}   holds, hence $A'\in \overline{\mathrm{G}}[U]$.
\end{proof}

\begin{lem}\label{lem:Gsubdelta}
Let $U$ be a countable set and $\mathrm{G}$ be a subgroup of $\mathfrak S(U)$.  Then $\overline{\mathrm{G}}[U]$  is  a $G_{\delta}$-set of $\powerset (U)$ equipped with the powerset topology. 
\end{lem}
\begin{proof} 
A base for the powerset topology consists of a set of subsets of the form  $\powerset(U)\llbracket F, E\rrbracket=\{A\subseteq U\mid \text{$F\subseteq A$ and $E\cap A=\emptyset$}\}$ with $F$ and $E$ ranging over finite subsets of $U$.  Let $\langle F\tr p\rangle$ be a type of $\mathrm{G}$. Then 
\[
O(\langle F\tr p\rangle):=\bigcup_{a\in F}\powerset(U)\llbracket \emptyset, \{a\} \rrbracket\cup \bigcup_{x\in \mathrm{G}\langle F\tr p\rangle}\powerset(U)\llbracket F\cup\{x\}, \emptyset\rrbracket
\]
is an open subset of $\powerset (U)$. Then the intersection of the subsets  $O(\langle F\tr p\rangle)$ of $\powerset (U)$  with $\langle F\tr p\rangle$ ranging over the countable set of types  of $\mathrm{G}$ is a $G_\delta$-set. It follows from Corollary \ref{cor:characterization} that 
\[
\bigcap_{\langle F\tr p\rangle\in \mathrm{Types}(\mathrm{G})} O(\langle F\tr p\rangle)=\overG[U]. 
\] 
\end{proof}

\begin{thm}\label{them:cardcopies} Let $U$ be a countable set and $\mathrm{G}$ be a subgroup of $\mathfrak S(U)$.  Then the cardinality of $\overline{\mathrm{G}}[U]$ is either $1$ or $2^{\aleph_0}$. 
\end{thm}
\begin{proof}
According to Lemma \ref{lem:isolated} if $\overline{\mathrm{G}}[U]$ has more than one element, it has no isolated point. Since  $\overline{\mathrm{G}}[U]$ is a $G_{\delta}$-subset of the Polish space $\powerset (U)$ it has the cardinality of the continuum. (See for example Theorem (6.2) of \cite{Kechris}.)
\end{proof}

\begin{lem}\label{lem:updirected}
Let $U$ be a countable set and $\mathrm{G}$ be a subgroup of $\mathfrak S(U)$.  Then $\overline{\mathrm{G}}[U]$   is closed under the union of non-empty up-directed families of members of $\overline{\mathrm{G}}[U]$. 
\end{lem}
\begin{proof} 
Let $\mathcal C:=(C_i; i\in I)$ with $I\not =\emptyset$ be an up-directed family of members of $\overline{\mathrm{G}}[U]$. Let $C:= \bigcup_{i\in I} C_i$. We claim that $C\in \overline{\mathrm{G}}[U]$. For that we apply Theorem \ref{thm:characterization}. Let  $F\in \finp{C}$ and  $x\in U\setminus F$. We have to show  that  there is a $g \in  \mathrm{G}\langle F\rangle$ such that  $g(x)\in C$. This is immediate. Since $\mathcal C$ is up-directed, there is an $i\in I$ such that $F\subseteq C_i$. Since $C_i\in \overline{\mathrm{G}}[U]$ there is a $g_i\in G\langle F\rangle$ such that $g_i(x)\in C_i$. Since $C_i\subseteq C$, $g_i(x)\in C$. It suffices to set $g:=g_i$. 
\end{proof}   

\begin{cor}\label{cor:maximal}
Let $U$ be a countable set and $\mathrm{G}$ be a subgroup of $\mathfrak S(U)$.  Then for every subset $F$ of $U$ and for every copy $C\in \overline{\mathrm{G}}[U]$ with $F\cap C=\emptyset$ there exists  a  maximal copy $C'\in \overline{\mathrm{G}}[U]$ with $C\subseteq C'$ and with  $F\cap C'=\emptyset$.  
\end{cor}

\begin{defin}\label{defin:invarcop}
Let $U$ be a countable set and $\mathrm{G}$ be a subgroup of $\mathfrak S(U)$. Let $f\in \overG$ and $h\in \overG$.  Then  $f(h):=f\circ h\circ f^{-1}$ and $f[\overG]:=\{f(h)\mid h\in \overG\}$ and $f[\mathrm{G}]:=\{f(g)\mid g\in \mathrm{G}\}$. 
\end{defin}

\begin{lem}\label{lem:invarcop}
Let $U$ be a countable set and $\mathrm{G}$ be a subgroup of $\mathfrak S(U)$. Let $f\in \overG$. Then:
\begin{enumerate}
\item $f(k)\circ f(h)=f(k\circ h)$ for all $\{k,h\}\subseteq \overG$.  
\item The function $f(g)$ is a bijection of $f[U]$ onto $f[U]$ with inverse $f(g^{-1})$ and $f[\mathrm{G}]$ is a subgroup of the symmetric group $\mathfrak{S}(f[U])$. 
\item $f: U\to f[U]$ is an isomorphism of the group action of $\mathrm{G}$  mapping every  $g\in \mathrm{G}$ to $f(g)$.  It is also an isomorphism of the monoid action of $\overG$  mapping every  $h\in \overG$ to $f(h)$.
 \item $C\in (f[\overG])[f[U]]$ if and only if $C\in \overG[U]$ and $C\subseteq f[U]$. 
\end{enumerate}
\end{lem}
\begin{proof}\\ 
\noindent
(1): Let $\{k,h\}\subseteq \overG$. Then $(f(k)\circ f(h))=f\circ k\circ f^{-1}\circ f\circ h\circ f^{-1}=f\circ (k\circ h)\circ f^{-1} = f(k\circ h)$. \\
\noindent
(2): Let $g\in \mathrm{G}$ and $x\in f[U]$. Clearly $f(g)\in (f[U])^{f[U]}$.  It follows from (1) that
   $f(g)\circ f(g^{-1})= f(g\circ g^{-1})=f(\mathrm{id}_U)=\mathrm{id}_{f[U]}$ and hence $f(g)\big(f(g^{-1})(x)\big)=x$.  Implying that the function $f(g)$ is a bijection of $f[U]$ onto $f[U]$ with inverse $f(g^{-1})$. Hence $f[\mathrm{G}]$ is a subgroup of the symmetric group $\mathfrak{S}(f[U])$.    \\
\noindent
(3): It remains to prove that if $\{k,h\}\subseteq \overG$ with $k\not=h$ then $f(k)\not=f(h)$. Let $x\in U$ with $k(x)\not=h(x)$. If $(f(k))(f(x))=(f(h))(f(x))$ then $(f\circ k\circ f^{-1})(f(x))=(f\circ k\circ f^{-1})(f(x))$ and $f\circ k(x)=f\circ h(x)$. The function $f$ is one-to-one on $U$ and hence the contradiction $k(x)=h(x)$. It follows that $f$ is one-to-one on $\overG$. \\
(4): Let $V=f[U]$ and  $C\in (f[\overG])[V]$. Then $C\subseteq V$ and  there exists  a function $h\in \overG$ with $f(h)[V]=C$. Hence $(f\circ h)[U]=(f\circ h\circ f^{-1})[V]=C$. Which in turn implies that $C\in \overG[U]$. Let then $C\in \overG[U]$ and $C\subseteq V$. Let $h\in\overG$ with $h[U]=C$. Let $F$ be a finite subset of $U$.  Let $k=f^{-1}\circ h$. There exists a function $g\in \mathrm{G}$ for which $\restrict{g}{F}=\restrict{h}{F}$. There exists a function $l\in \mathrm{G}$ with $\restrict{f}{(f^{-1}(h[F]))}=\restrict{l}{(f^{-1}(h[F]))}$. Then $\restrict{k}{F}=\restrict{(l^{-1}\circ g)}{F}$. We conclude that $k\in \overG$ and that $f(k)[f[U]]=f(k)[V]=(f\circ k\circ f^{-1})[V]=(h\circ f^{-1})[V]=h[U]=C$. 

\end{proof}

Item (4) of Lemma \ref{lem:invarcop} states that a subset $S$ of a copy $C=f[U]$ for $f\in \overG[U]$  is a copy in respect to the action of $\mathrm{G}$ on $U$ if and only if it is also a copy in respect to the action of $f[G]$ on $C$. That is being a copy is "invariant`` of that interpretation.

\begin{lem}\label{lem:invpropacttp}
Let $U$ be a countable set and $\mathrm{G}$ be a subgroup of $\mathfrak S(U)$.  Let $f\in \overG$ and   $F\cup \{x,y\}$ be a finite subset of $f[U]$.   Then: 
\begin{enumerate}
\item The types $\langle F\tr x\rangle$ and $\langle F\tr y\rangle$ are equal with respect to the group action of $\mathrm{G}$ on $U$  if and only if they are equal with respect to the group action of the group $f[\mathrm{G}]$ on $f[U]$. 
\item If $F\cup \{x\}\subseteq f[U]:=C$ for $f\in \overG$, then: 
\begin{enumerate}
\item $\mathrm{G}\langle F\tr x\rangle\cap C=(f[\mathrm{G}])\langle F\tr x\rangle$. 
\item If $\mathrm{G}\langle F\tr x\rangle$ is finite then $\mathrm{G}\langle F\tr x\rangle=(f[\mathrm{G}])\langle F\tr x\rangle \subseteq C$.
\end{enumerate} 
\end{enumerate}
\end{lem}
\begin{proof}\\
(1): Let $F\cup \{x,y\}\subseteq f[U]:=C$ for $f\in \overG$. Assume $\langle F\tr x\rangle=\langle F\tr y\rangle$. Then there exists a function $g\in \mathrm{G}\langle F\rangle$ with $g(x)=y$. Let the function  $k\in \mathrm{G}$ be such that $\restrict{k}{(f^{-1}[F\cup \{x,y\}])}=\restrict{(f^{-1}\circ g\circ f)}{(f^{-1}[F\cup \{x,y\}])}$. Then $k\in \mathrm{G}\langle f^{-1}[F]\rangle$ and $k(f^{-1}(x))=f^{-1}(y)$ and $\restrict{(f\circ k\circ f^{-1})}{(F\cup \{x,y\})}=\restrict{g}{(F\cup \{x,y\})}$. Hence, the function $f(k)$ is a witness for $\langle F\tr x\rangle=\langle F\tr y\rangle$ in respect to the group action of $f[G]$ on $f[U]$.  On the other hand let $k\in \mathrm{G}$ for which the function $f(k)$  is a witness for $\langle F\tr x\rangle=\langle F\tr y\rangle$ in respect to the group action of $f[G]$ on $f[U]$. Then $f(k)(a)=a$ for all $a\in F$ and $f(k)(x)=y$. Let $g\in \mathrm{G}$ be such that $\restrict{g}{(F\cup \{x,y\})}=\restrict{f(k)}{(F\cup \{x,y\})}$. Then  the function $g$ is a witness for $\langle F\tr x\rangle=\langle F\tr y\rangle$ in respect to the group action of $\mathrm{G}$ on $U$.

Item (2)(a) follows from Item (1). Item (2)(b) follows from Item~(2)(a) and Corollary \ref{cor:maptypef}.  
\end{proof}

\section{Algebraic closure}\label{sect:algebcl}

 Let $U$ be a countable set and $\mathrm{G}$ be a subgroup of $\mathfrak S(U)$. For $F$ a finite subset of $U$ we defined in Section \ref{sect:Torb} the set $\mathrm{Orb}(\mathrm{G}\langle F\rangle)$ to be the set of typesets with sockel $F$, that is the set of orbits of the group $\mathrm{G}\langle F\rangle$.   Let $F\subseteq U$, not necessarily finite.  The {\em algebraic closure of $F$} is the set, denoted  $\mathfrak{ac}(F)$, and defined by setting: 
\[
\mathfrak{ac}(F)=
\begin{cases}
\bigcup\{T\in \mathrm{Orb}(\mathrm{G}\langle F\rangle)\mid \text{$T$ is finite}\},    &\text{if $\bullet:$  $F$ is finite;}\\
 \bigcup\{\mathfrak{ac}(F')\mid F'\in \finp{F}\},        &\text{$\bullet \bullet:$ in general.}
\end{cases}
\]
Definition $\bullet \bullet$ is consistent with Definition $\bullet$ because if a typeset $T$ of a type  is finite then  every continuation of this type has a typeset which is a subset of $T$ and hence finite. (Fact \ref{fact:partitiont}.)  Finally, set $K_{\mathfrak{ac}}= \mathfrak{ac}(\emptyset)$. Note that $x\in K_{\mathfrak{ac}}$ if and only if the set $\mathrm{G}\langle \emptyset\tr x\rangle$, that is  the orbit of $\mathrm{G}$ containing $x$ is finite. 

\begin{lem}\label{lem:acfinoligo}
If $\mathrm{G}$ is oligomorphic then $\mathfrak{ac}(F)$ is finite for every finite subset $F$ of $U$. 
\end{lem}
\begin{proof}
By Definition \ref{defin:oligomorphic3}, if $\mathrm{G}$ is oligomorphic the set $\mathrm{Orb}(\mathrm{G}\langle F\rangle)$ of typesets with sockel $F$ is finite. Hence $\mathfrak{ac}(F)$ is then a finite union of finite sets. 
\end{proof} 

For $\mathrm{R}$ a relational structure let $\mathrm{Age}(\mathrm{R})$, the \emph{age} of $\mathrm R$,  denote the class of structures which are isomorphic to a finite induced substructure of $\mathrm{R}$. 
A  group $\mathrm{G}$ is \emph{algebraically finite} if $\mathfrak{ac}(F)$ is finite for every finite subset $F$ of $U$. There are many algebraically finite groups which are not oligomorphic. In fact there are many non oligomorphic groups such that $\mathfrak{ac}(F)=F$ for every finite $F$. If, in this case, a structure $\mathrm{R}$  is a countable homogeneous relational structure such that $\aut (R)= \overG\cap \mathfrak{S}$ then this amounts to the fact that $\mathrm{Age}(\mathrm{R})$  satisfies the \emph{disjoint amalgamation property} (DAP). See \cite{Fra}.  

\begin{lem} \label{lem:acl-kernel}
Let $x\in U$. Then $x\in K_{\mathfrak{ac}}$ if and only if  there exists a finite subset $F$ of $U$ such that $x\in g[F]$ for every function  $g\in \mathrm{G}$. 
\end{lem}
\begin{proof}
Suppose that $x\in K_{\mathfrak{ac}}$. Set $F:=\mathrm{G}\langle \emptyset\tr x\rangle$,  the orbit of $x$ w.r.t. the action of $\mathrm{G}$. By definition, $F$ is finite and    $g[F]=F$ for every $g\in \mathrm{G}$.  

Conversely, let $F\in \finp{U}$ of minimal cardinality for which $x\in g[F]$ for every function  $g\in \mathrm{G}$. This means  that $x\in F':= \bigcap\{ g[F]: g\in \mathrm{G}\}$. We have $h[F']= \bigcap\{ h\circ g[F]: g\in \mathrm{G}\} \supseteq  F'$ for every $h\in G$ and  since $F'$ is finite, $h[F']=F'$. Since $x\in F'$, the orbit of $x$  w.r.t. the action of $\mathrm{G}$ is included into $F'$. By minimality, this orbit, $F'$ and $F$ are equal. Hence, $x\in K_{\mathfrak{ac}}$. \end{proof}

Let $\mathrm{R}$ be a relational structure with domain $U$. The \emph{kernel} of $\mathrm{R}$ is the set $\mathrm{Ker}(\mathrm{R}):=\{x\in U: \mathrm{Age} (\mathrm{R}_{-x})\not = \mathrm{Age}(\mathrm{R})\}$. ($\mathrm{R}_{-x}:=\restrict{\mathrm{R}}{(U\setminus\{x\})}$.)

\begin{cor}
Let $\mathrm{R}$ be a homogeneous relational structure and $\mathrm{G}= \aut(\mathrm{R})$. Then $K_{\mathfrak{ac}}= \mathrm{Ker}(\mathrm{R})$. 
\end{cor}
\begin{proof}Let $x\in \mathrm{Ker}(\mathrm{R})$. By definition there exists a finite subset $F$ of $U$ containing $x$ such that $\restrict{\mathrm{R}}{F}$ does not embed into $\mathrm{R}_{-x}$. It follows that $x\in g[F]$ for every $g\in \mathrm{G}$. According to Lemma \ref{lem:acl-kernel}, the element $x\in K_{\mathfrak{ac}}$. Conversely, suppose that $x\in K_{\mathfrak{ac}}$. Then $F:= \mathrm{G}\langle \emptyset\tr x\rangle$ is finite. Since $F$ is preserved by every member of $\mathrm{G}$, there is no embedding of $\restrict{\mathrm{R}}{F}$   into $\mathrm{R}_{-x}$.  
\end{proof}
 
\begin{lem}
The map $F\rightarrow \mathfrak{ac}(F)$ is an algebraic closure operator. 
\end{lem}

\begin{proof}
We have trivially $F\subseteq \mathfrak{ac}(F)\subseteq \mathfrak{ac} (F')$ for every $F\subseteq F'$. In particular $\mathfrak{ac}(F)\subseteq \mathfrak{ac}(\mathfrak{ac}(F))$. We have $\mathfrak{ac}(F) =\mathfrak{ac}(\mathfrak{ac}(F))$. Indeed, let $x\in \mathfrak{ac}(\mathfrak{ac}(F))$ Then by definition there is some finite $Q\subseteq \mathfrak{ac}(F)$ such that $x\in \mathfrak{ac}(Q)$. Select for each $q\in Q$ some finite $Q_q\subseteq F$ such that $q\in \mathfrak{ac}(Q_q)$.  Set $S:= \bigcup_{q\in Q}Q_q$. Then $x\in \mathfrak{ac}(S)\subseteq \mathfrak{ac}(F)$. Thus, $\mathfrak{ac}$ defines a closure operator. According to the second clause in the definition, it is algebraic.  
\end{proof} 

This operator does not necessarily have  the exchange property. Here is an example where the group is transitive. 
Let $U:= [\N]^2$ and $\mathrm{G}$ be the group of permutations of $U$ induced by permutations of $\N$.  Let $u:= \{x,y\}\in U$, $u':= \{x',y'\}\in U$ be two disjoint pairs of $\N$ and $v:= \{ x,x'\}$. Then $v\not\in \mathfrak{ac}(\{u\})= \{u\}$, and $v\in \mathfrak{ac}(\{u,u'\}$ while $u'\not \in \mathfrak{ac}(\{u, v\})$.

\section{Ranked closure}\label{sect:rankedcl}

 Let $U$ be a countable set and $\mathrm{G}$ be a subgroup of $\mathfrak S(U)$.   A type $\langle F\tr p\rangle$ has {\em rank} 0 if  $\mathrm{G}\langle F\tr p\rangle$ is finite. Let $\mathfrak{R}_0$ be the set of types $\langle F\tr p\rangle$ having rank 0. We define recursively sets $\mathfrak{R}_\alpha$ of types for $\alpha$ an ordinal. If for every ordinal $\beta<\alpha$ the set $\mathfrak{R}_\beta$ has been determined then the type $\langle F\tr p\rangle$ is an element of $\mathfrak{R}_\alpha$ if there exists a finite subset $F'\supseteq F$ of $U$ so that every continuation  $\langle F'\tr q\rangle$ of $\langle F\tr p\rangle$ is a type in $\mathfrak{R}_\beta$ for some ordinal $\beta<\alpha$. Note that  $\mathfrak{R}_\beta\subseteq \mathfrak{R}_\alpha$ for all $\beta < \alpha$; letting $F'=F$. A type $\langle F\tr p\rangle$ has {\em rank} $\alpha$ if $\langle F\tr p\rangle\in \mathfrak{R}_\alpha$ and $\langle F\tr p\rangle\not\in \mathfrak{R}_\beta$ for any $\beta<\alpha$. Hence if a type $\langle E\tr x\rangle \in \mathfrak{R}_\gamma$, then the rank of the type $\langle E\tr x\rangle$ is less than or equal to $\gamma$.   A type $\langle F\tr p\rangle$ is {\em unranked} if there is no ordinal $\alpha$ with $\langle F\tr p\rangle\in \mathfrak{R}_\alpha$.  

\vskip 5pt

Note that a type $\langle F\tr p\rangle$ is unranked if and only if:

\noindent 
\textbf{(p)}: For every finite set $F'$ with $F\subseteq F'\subseteq U$ exists a continuation\\
\phantom{.} \hskip 21pt $\langle F' \tr q\rangle$ of $\langle F\tr p\rangle$ which is unranked. 

\begin{lem}\label{lem:contrankrank}
If a type $\langle F\tr p\rangle$ is ranked and has rank $\alpha$ then every continuation $\langle E\tr q\rangle$ of the type $\langle F\tr p\rangle$ is ranked and has rank less than or equal to $\alpha$.  
\end{lem}
\begin{proof}
By induction on the rank of the type  $\langle F\tr p\rangle$. If $\mathrm{G}\langle F\tr p\rangle$ is finite then $\mathrm{G}\langle E\tr q\rangle \subseteq \mathrm{G}\langle F\tr p\rangle$ is finite.  Let $\alpha>0$. If the continuation $\mathrm{G}\langle E\tr q\rangle$ is finite  then $\langle E\tr q\rangle$ is ranked and has rank $0<\alpha$. Otherwise we procede  as follows:   There exists a finite subset $F'\supseteq F$ of $U$ so that the rank $\beta(x)$ of every continuation  $\langle F'\tr x\rangle$ of $\langle F\tr p\rangle$ is smaller than $\alpha$.  The set $E'=E\cup F'$ is finite. The set of typesets $\mathrm{G}\langle E'\tr x\rangle$ for $x\in \mathrm{G}\langle E\tr q\rangle$  forms a partition of the typeset $\mathrm{G}\langle E\tr q\rangle$. Each one of the types $\langle E'\tr x\rangle$ is a continuation of the type $\langle F'\tr  x\rangle$ having rank $\beta(x)<\alpha$. Implying that the type $\langle E\tr q\rangle\in \mathfrak{R}_\alpha$ and hence that the rank of the type $\langle E\tr q \rangle$ is less than or equal to $\alpha$.    
\end{proof}

\begin{cor}\label{cor:contrankrank}
Let $\langle F\tr p\rangle$ be a ranked type and $F\subseteq E\in\finp{U}$. Then the set $\mathrm{G}\langle F\tr p\rangle$ is  the union over the set of typesets of the  ranked types $\langle E\tr x\rangle$  with $x\in \mathrm{G}\langle F\tr p\rangle$. 

\end{cor}

\begin{lem}\label{lem:rankext}
Let $\langle F\tr p\rangle$ be a type and $g\in \mathrm{G}$. Then the type $\langle F\tr p\rangle$ is ranked if and only if the type $\langle g[F]\tr g(p)\rangle$ is ranked. The type $\langle F\tr p\rangle$ has rank $\alpha$  if and only if the type $\langle g[F]\tr g(p)\rangle$ has rank $\alpha$. 
\end{lem}
\begin{proof}
By induction on the rank of the type  $\langle F\tr p\rangle$.  It follows from Corollary \ref{cor:maptypef} that the type $\langle F\tr p\rangle$ has rank 0 if and only if the type $\langle g[F]\tr g(p)\rangle$ has rank 0. Let the type $\langle F\tr p\rangle$ have rank $\alpha>0$. Then there exists a finite $F'\supseteq F$ for which every continuation $\langle F'\tr q\rangle$ of the type $\langle F\tr p\rangle$ has a rank, say $\beta_q$,  which is  smaller than $\alpha$. Also $g[F']\supseteq g[F]$. Let $\langle g[F']\tr s\rangle$ be a continuation of the type   $\langle g[F]\tr g(p)\rangle$. Then $s\in \mathrm{G}\langle g[F]\tr g(p)\rangle$. It follows,  using Lemma \ref{lem:maptypef}, that $g^{-1}(s)\in \mathrm{G}\langle F\tr p\rangle$ and hence that the type $\langle F'\tr g^{-1}(s)\rangle$ is a continuation of the type $\langle F\tr p\rangle$. It has a rank $\beta_{g^{-1}(s)}<\alpha$. Then, via the induction, the rank of the type $\langle g[F']\tr s\rangle$ is also equal to $\beta_{g^{-1}(s)}<\alpha$. Hence the rank of the type $\langle g[F]\tr g(p)\rangle$ is equal to $\bar{\alpha}$ less than or equal to $\alpha$. Applying $g^{-1}$ to the type $\langle g[F]\tr g(p)\rangle$ we obtain $\alpha\leq \bar{\alpha}$. Hence  $\langle F\tr p\rangle$ is ranked  if and only if $\langle g[F]\tr g(p)\rangle$ is ranked and both types have the same rank. 
\end{proof}

\begin{lem}\label{lem:contall}
Let $\langle F\tr p\rangle$ be a ranked type. Then $\mathrm{G}\langle F\tr p\rangle\subseteq C$ for every $C\in \overG[U]$ with $F\subseteq C$.   
\end{lem}
\begin{proof}
Let $F\subseteq C\in \overG[U]$.  According to Lemma \ref{lem:fix-copy} there is a function  $h\in \overG\langle F\rangle$ with $h[U]=C$. We proceed via induction on the rank of the type  $\langle F\tr p\rangle$.  In the case that the typeset   $\mathrm{G}\langle F\tr p\rangle$ is finite the Lemma follows from Corollary \ref{cor:maptypef}.

 Let $0<\alpha$ be an ordinal for which $\mathrm{G}\langle F''\tr q\rangle\subseteq C$ for all types $(F''\tr q)$ of rank $\beta<\alpha$ and with $F''\subseteq C$.  Let $\langle F\tr p\rangle$ be a type of rank $\alpha$.  There exists then a finite subset $F'\subseteq U$ with  $F'\supseteq F$ and  so that  every continuation $\langle F'\tr q\rangle$ of $\langle F\tr p\rangle$  has a rank, say $\beta(q)$, which is smaller than $\alpha$. Let $Q:=q\in \mathrm{G}\langle F\tr p\rangle$ be the set of those elements $q$. It follows from Fact \ref{fact:partitiont} that the sets $\mathrm{G}\langle F'\tr q\rangle$ for $q\in Q$ form a partition of the set $\mathrm{G}\langle F\tr p\rangle$. Let $g\in \mathrm{G}$ be such that $\restrict{h}{F'}=\restrict{g}{F'}$.  Then $g[F']\subseteq C$ and   $g\in \mathrm{G}\langle F\rangle$. Consequently the sets $\mathrm{G}\langle g[F']\tr g(q)\rangle$  for $q\in Q$ form a partition of the set $\mathrm{G}\langle F\tr p\rangle$. Each of the types $\langle g[F']\tr g(q)\rangle$ is ranked and its rank is equal to $\beta(q)$ according to Lemma \ref{lem:rankext}. Then $\mathrm{G}\langle h[F']\tr  h(q)\rangle\subseteq C$ using the induction assumption. Implying that $\mathrm{G}\langle F\tr p\rangle\subseteq C$.  

\end{proof}

Let $\{f,h\}\subseteq  \overG$ and $C=f[U]=h[U]$. Then the group $f[G]\subseteq \mathfrak{S}(f[U])$ acts on the set $f[U]$. See Definition \ref{defin:invarcop} and Lemma \ref{lem:invarcop}.  It follows from Lemma  \ref{lem:invpropacttp}  that for every finite $F\subseteq C$ and $x\in C\setminus F$ the type $\langle F\tr x\rangle$ with respect to the group action $f[\mathrm{G}]$ is the same as with respect to the group action $h[\mathrm{G}]$. Hence we can define:

\begin{defin}\label{defin:typunderf}
Let $\langle F\tr p\rangle$ be a type and $f\in \overG$ with $F\cup \{p\}\subseteq f[U]$:=C. Then  $\langle F\tr p\rangle_C$ denotes this type with sockel $F$ in respect to  the action of the group $f[\mathrm{G}]$.   
\end{defin}
Note that, for every $f\in \overG$,  the typeset of the type $\langle F\tr p\rangle_C$ is the set
\[
\mathrm{G}\langle F\tr p\rangle\cap C=(f[\mathrm{G}])\langle F\tr p\rangle_C=\{(f(g))(x)\mid g\in \mathrm{G}\}. 
\]

\begin{lem}\label{lem:contallinvariant}
Let  $f\in \overG$ with $f[U]:=C$. Then a type $\langle F\tr p\rangle$, with $f\cup \{p\}\subseteq C$,  is ranked if and only if the type $\langle F\tr p\rangle_C$ is ranked with respect to the group action $f[\mathrm{G}]$. The type $\langle F\tr p\rangle$ has rank $\alpha$  if and only if the type $\langle F\tr p\rangle_C$ has rank $\alpha$ with respect to the group action $f[\mathrm{G}]$. 
\end{lem}
\begin{proof}
Let $\mathrm{rank}\langle F\tr p\rangle$ denote the rank of a type with respect to the group action $\mathrm{G}$ and $\mathrm{rank}_f\langle F\tr p\rangle_C$ denote the rank of a type in $C$ with respect to the group action $f(\mathrm{G})$. We proceed via induction on the ranks.  In the case that the typeset   $\mathrm{G}\langle F\tr p\rangle$ is finite the Lemma follows from Lemma \ref{lem:invpropacttp} Item (2)(b).   

Let $\alpha$ be the smallest ordinal for which there exists a type $\langle F\tr p\rangle_C$ with  $\nu:=\mathrm{rank}\langle F\tr p\rangle\not=\alpha:=\mathrm{rank}_f\langle F\tr p\rangle_C$. There exists  a finite subset $F'\subseteq C$ with  $F'\supseteq F$ and  so that  every continuation $\langle F'\tr q\rangle_C$ of $\langle F\tr p\rangle_C$  has a rank, say $\beta(q)=\mathrm{rank}_f\langle F'\tr q\rangle_C$, which is smaller than $\alpha$. The types $\langle F'\tr q\rangle$ are then continuations of the type $\langle F\tr p\rangle$ and via induction have the ranks $\beta_q$. Hence $\nu\leq \alpha$. There exists  a finite subset $F'\subseteq U$ with  $F'\supseteq F$ and  so that  every continuation $\langle F'\tr q\rangle$ of $\langle F\tr p\rangle$  has a rank, say $\beta(q)=\mathrm{rank}\langle F'\tr q\rangle$, which is smaller than $\nu$. Let $Q:=q\in \mathrm{G}\langle F\tr p\rangle\setminus F'$ be the set of those elements $q$. It follows from Fact \ref{fact:partitiont} that the sets $\mathrm{G}\langle F'\tr q\rangle$ for $q\in Q$ form a partition of the set $\mathrm{G}\langle F\tr p\rangle$. According to Lemma~\ref{lem:fix-copy} there is a function  $h\in \overG\langle F\rangle$ with $h[U]=C$.  Let $g\in \mathrm{G}$ be such that $\restrict{h}{F'}=\restrict{g}{F'}$.  Then $g[F']\subseteq C$ and   $g\in \mathrm{G}\langle F\rangle$. Consequently the sets $\mathrm{G}\langle g[F']\tr g(q)\rangle$  for $q\in Q$ form a partition of the set $\mathrm{G}\langle F\tr p\rangle$. Each of the types $\langle g[F']\tr g(q)\rangle$ is ranked and its rank is equal to $\beta(q)$ according to Lemma \ref{lem:rankext}. Implying according to  Lemma~\ref{lem:contall} that  $\mathrm{G}\langle h[F']\tr  h(q)\rangle\subseteq C$. Via induction then $\mathrm{rank}\langle g[F']\tr g(q)\rangle_C=\beta(q)<\nu$. Hence $\alpha\leq \nu$. 

\end{proof}

\begin{lem}\label{lem:unrunion2}
If  $\mathrm{G}\langle F\tr x\rangle$ is an unranked type,   then there exists an element $C\in  \overG[U]$ with $F\subseteq C$  and with $x\not\in C$. 
\end{lem}
\begin{proof}
Let  $(u_i\tr i\in \omega)$ be an enumeration of $U\setminus F$. Let $U_n=\{u_i\mid i\in n\}$.  Note that $U_0=\emptyset$. We will  construct a sequence $\{x_n\in U\mid n\in \omega\}$
so that for all $n\in \omega$: 
\begin{enumerate}
\item $x_0=x$. 
\item $x_n\in U\setminus(F\cup U_n)$ and the type $\langle F\cup U_n\tr x_n\rangle$ is unranked.  
\item The type $\langle F\cup U_{n+1}\tr x_{n+1}\rangle$ is a continuation of the type \\$\langle  F\cup U_n\tr x_n\rangle$.  
\end{enumerate}
The type $\langle F\cup U_n\tr x_n\rangle$ is unranked. Hence there exists, according to property \textbf{(p)},   an unranked type $\langle F\cup  U_{n+1}\tr x_{n+1}\rangle$ which is a continuation of the type $\langle F\cup  U_n\tr x_n\rangle$. Let $f_n\in \mathrm{G}\langle F\cup  U_n\rangle$ be the function with $f_n(x_{n+1})=x_{n}$.     

Let $g_0$ be the identity  function on $U$, that is $\{g_0\}=\mathrm{G}\langle U\rangle$. We will  construct recursively functions $g_n\in \mathrm{G}\langle F\rangle$ for $n\in \omega$   for which $g_n(x_n)=x$ and for which $l_n:=\restrict{g_n}{(F\cup U_n)}\subseteq \restrict{g_{n+1}}{(F\cup U_{n+1})}=l_{n+1}$.  Assume that $g_n$ has been constructed. Let $g_{n+1}=g_n\circ f_n$. Note that $x\not\in l_n[U_n]$ for all $n\in \omega$.

Then $l_0\subset l_1\subset l_2\subset\dots$ and $l:=\bigcup_{n\in \omega}l_n\in \overline{\mathrm{G}}$ with $x\not\in l[U]$. Let $C=l[U]$. 
\end{proof}

\begin{defin}\label{defin:ranktyp}
For $F\subseteq U$ let the {\em intersection closure}:
\[
\mathfrak{ic}_\mathrm{G}(F):=\bigcap\{C\in \overG[U]\mid F\subseteq C\}= \bigcap\overG[U]\llbracket F\rrbracket.  
\]
For $F$ a finite subset of $U$ let $\mathfrak{rc}_\mathrm{G}(F)$, the {\em ranked closure of $F$},  be the union of $F$ together with the union of the typesets $\mathrm{G}\langle F\tr p\rangle$ for which $\langle F\tr p\rangle$ is a ranked type. (In particular then $\mathfrak{rc}_\mathrm{G}(\emptyset)$ is the union of the ranked orbits of $\mathrm{G}$.)

In general, for a subset $S$ of $U$ let:
\begin{align}\tag{1}   
\mathfrak{rc}_\mathrm{G}(S):=\bigcup\{\mathfrak{rc}(F)\mid F\in \finp{S}\}.  
\end{align} 
If the group $\mathrm{G}$ is understood we write $\mathfrak{ic}$ for $\mathfrak{ic}_\mathrm{G}$ and $\mathfrak{rc}$ for $\mathfrak{rc}_\mathrm{G}$
\end{defin}
It follows from Corollary \ref{cor:contrankrank} that the ranked closure definition (1) is consistent with the ranked closure definition for finite subsets $F$ of $U$. 

\begin{note}\label{note:rcequivalgoli}
Let the group $\mathrm{G}$ be oligomorphic, then: Every ranked type has rank 0. A type is unranked if and only if its typeset is infinite. The ranked closure of every finite set is finite. For every set $S\subseteq U$ is $\mathfrak{ac}(S)=\mathfrak{rc}(S)$. 
\end{note}

\begin{thm}\label{thm:rcclinters}
 Let $U$ be a countable set and $\mathrm{G}$ be a subgroup of $\mathfrak S(U)$. Let $S\subseteq U$. Then $\mathfrak{rc}_\mathrm{G}(S)\subseteq \mathfrak{ic}_\mathrm{G}(S)$. If $F$ is a finite subset of $U$ then 
$\mathfrak{rc}_\mathrm{G}(F)= \mathfrak{ic}_\mathrm{G}(F)$. 
\end{thm}
\begin{proof}
The inclusion $\mathfrak{rc}_\mathrm{G}(S)\subseteq \mathfrak{ic}_\mathrm{G}(S)$ follows from Lemma \ref{lem:contall}. The inclusion $\mathfrak{ic}_\mathrm{G}(F)\supseteq \mathfrak{rc}_\mathrm{G}(F)$ for finite $F$ follows from Lemma \ref{lem:unrunion2}. 
\end{proof}

According to the definition ot typeset, a subset $T$ of $U$ with $t\in T$ is the point orbit of $\mathrm{G}$ containing $t$ if and only if $T=\mathrm{G}\langle \emptyset\tr t\rangle$. Since by Theorem \ref {thm:rcclinters}, $\mathfrak{rc}_\mathrm{G}(\emptyset)= \mathfrak{ic}_\mathrm{G}(\emptyset)$, the intersection $\mathfrak{ic}(\emptyset)=\bigcap\overG(U)$  of all copies   for the group $\mathrm{G}$ is equal to the union of the point orbits of $\mathrm{G}$. In particular,  we have a characterization in terms of ranked closure of the nonexistence of a proper copy (compare with Lemma \ref{lem:justonecopy}). 

\begin{cor}\label{nopropercopy}
Let $U$ be a countable set and $\mathrm{G}$ be a subgroup of $\mathfrak S(U)$. Then there is no proper copy (i.e., $\overG[U]=\{U\}$, or equivalently, if every function $f\in \overline{\mathrm{G}}$ is bijective) if and only if every point orbit of the group $\mathrm{G}$ is  ranked. 
\end{cor}

\begin{cor}\label{cor:rcclinters}
 Let $U$ be a countable set and $\mathrm{G}$ be a subgroup of $\mathfrak S(U)$ and $F$ a finite subset of $U$ for which  the set  $\mathfrak{rc}(F)$ is finite. Then $\overG[U]\llg \mathfrak{rc}(F)\rrg=\overG[U]\llg F\rrg$.    
 \end{cor}
 \begin{proof}
Clearly $\overG[U]\llg \mathfrak{rc}(F)\rrg\subseteq \overG[U]\llg F\rrg$. It follows from Theorem \ref{thm:rcclinters} that $\overG[U]\llg \mathfrak{rc}(F)\rrg\supseteq \overG[U]\llg F\rrg$.
 
 \end{proof}

\begin{lem}\label{lem:rcclosedop}
 Let $U$ be a countable set and $\mathrm{G}$ be a subgroup of $\mathfrak S(U)$. Then the map $\mathfrak{rc}: \powerset{(U)}\to \powerset{(U)}$, which maps every subset $S$ of $U$ to the subset $\mathfrak{rc}(S)$ of $U$,  is an algebraic closure operator.
\end{lem}
\begin{proof}
If $F\subseteq F'$ are two finite subsets of $U$ then $F\subseteq \mathfrak{rc}(F)\subseteq \mathfrak{rc} (F')$  follows trivially from  Definition \ref{defin:ranktyp} and then from Corollary \ref{cor:contrankrank}. Hence  $S\subseteq \mathfrak{rc}(S)\subseteq \mathfrak{rc} (S')$ for every $S\subseteq S'$ using Definition \ref{defin:ranktyp} again.  In particular $\mathfrak{rc}(S)\subseteq \mathfrak{rc}(\mathfrak{rc}(S))$. We have to verify $\mathfrak{rc}(S) =\mathfrak{rc}(\mathfrak{rc}(S))$.  Let $x\in \mathfrak{rc}(\mathfrak{rc}(S))$.  Then by definition there there exists a   finite set  $F\subseteq \mathfrak{rc}(S)$ such that $x\in \mathfrak{rc}(F)$. Select for each $y\in F$ a finite set $F_y\subseteq S$ such that $y\in \mathfrak{rc}(F_y)$. Then $y\in\mathfrak{ic}(F_y)$ for every $y\in F$ according to Theorem \ref{thm:rcclinters}.   Set $E:= \bigcup_{y\in F}F_y$. Then $y\in\mathfrak{ic}(F_y)\subseteq  \mathfrak{ic}(E)$ for every $y\in F$, hence $F\subseteq \mathfrak{ic}(E)$.   Implying, using Theorem \ref{thm:rcclinters}, that $F\cup \{x\}\subseteq \mathfrak{ic}(E)$ and $x\in \mathfrak{rc}(E)\subseteq \mathfrak{rc}(S)$.  Thus, the map $\mathfrak{rc}$ defines a closure operator. According to stipulation (1) of  Definition \ref{defin:ranktyp}, the map $\mathfrak{rc}$ is algebraic.  

\end{proof}

However, there are examples for which the ranked closure of the empty set is empty and still there are no pairwise disjoint copies. 

\begin{example} 
Let $T_{\mathbb Z}$ be the ordered tree on a countable set $U$ in which each interval is finite, every element has countably many successors and a unique predecessor. Let $\mathrm{G}:= \aut(T_{\mathbb Z})$. Then: 

The powerset $\powerset(\omega)$ is embeddable in $\overline{\mathrm{G}}[U]$. Any  two copies pairwise intersect. For $x\in U$ is the set $\downarrow\!\!{x}:=\{y\in U\mid y\leq x\}$ the  intersection of all copies containing $x$.   
\end{example}

\subsection{The case of linear orders}\label{subs:linordcas}

Let $\mathrm{Q}=(Q;\leq)$ be the linear order of the rationals. For $S\subseteq \omega$ associate the copy $S_Q$ of $\mathrm{Q}$ given by:
\[
S_Q:=((-1,0)\cap Q)\cup \bigcup_{s\in S}((s,s+1)\cap Q). 
\]
Because the binary relational structure $\mathrm{Q}$ is homogeneous, see \cite{Fra}, it follows that $S_Q$ is a copy for the group $\mathrm{Aut}(\mathrm{Q})$ and hence that the partial order $(\overline{\mathrm{Aut}(\mathrm{Q})}[Q];\subset)$ of copies for the group $\overline{\mathrm{Aut}(\mathrm{Q})}$ embeds the partial order $(\powerset(\omega);\subset)$.

Let $\mathrm{U}=(U;\leq)$ be a linear order on a countable set $U$ and let $\mathrm{G}$ be the group of automorphisms of $\mathrm{U}$.  Clearly, every finite orbit of $\mathrm{G}$ is a singleton set, that is a linear order of order type \textbf{1}. Let $\zeta$ be the order type of the integers and note that $\zeta^0$ is the order type \textbf{1}.  If $\mathrm{U}$ is the natural linear order of the set of integers then all of the orbits of the group $\mathrm{G}\langle\{0\}\rangle$ are singleton sets. That is $\mathrm{Orb}(\mathrm{G}\langle\{0\}\rangle)=\{\{i\}\mid i\in U\}$. Implying that if an orbit of the group of automorphisms $\mathrm{G}$ of a linear order $\mathrm{U}$ has order type $\zeta$ then the rank of this orbit is 1. If  an orbit of the group of automorphisms $\mathrm{G}$ of a linear order $\mathrm{U}$ has order type $\zeta^2$ then the rank of this orbit is 2.   Consulting the discussion on powers of linear orders of order type $\zeta$ and in particular Corollary 8.6 of the Rosenstein book \cite{Rosenstein}, we obtain using Lemma \ref{lem:fix-copy}   and Theorem~\ref{thm:rcclinters}: 

\begin{thm}\label{thm:linordcas}
Let $\mathrm{U}=(U;\leq)$ be a linear order on a countable set $U$ and let $\mathrm{G}$ be the group of automorphisms of $\mathrm{U}$,then: 
\begin{enumerate}
\item Every orbit of $\mathrm{G}$  has order type $\zeta^\gamma$ or $\zeta^\gamma{\cdotp}\eta$ for some ordinal number $\gamma$. (The ordinal $\gamma$ is the Hausdorff rank of the scattered linear order $\zeta^\gamma$.)
\item There exists only one copy for $\mathrm{G}$ if and only if every orbit of $\mathrm{G}$ has order type $\zeta^\gamma$ for some ordinal number $\gamma$. The rank of such an orbit  is then the ordinal $\gamma$. 
\item If the group $\mathrm{G}$ has an orbit of order type $\zeta^\gamma{\cdotp}\eta$  then the partial order $(\overG[U];\subseteq)$ of copies for the group $\mathrm{G}$ embeds the partial order $(\powerset(\omega);\subseteq)$. 
\end{enumerate} 
\end{thm}

\section{An application to algebraically finite groups}

 \begin{lem}\label{lem:oenonfinempt}
  Let $U$ be a countable set and $\mathrm{G}$ be a subgroup of $\mathfrak S(U)$. If $\mathfrak{rc}(F)=U$ for some  finite subset $F$ of $U$ then $\mathfrak{rc}(E)=U$ for all finite subsets $E$ of $U$. 
 \end{lem}
 \begin{proof}
 Assume  that there exists a finite subset $F$ of $U$ with $\mathfrak{rc}(F)=U$. Then according to Theorem  \ref{thm:rcclinters} and Definition \ref{defin:ranktyp} and Lemma \ref{lem:fix-copy}:   $U=\mathfrak{ic}(F)=\bigcap\{C\in \overG[U]\mid F\subseteq C\}=\bigcap\overline{\mathrm{G}\langle F\rangle}[U]$. Hence $\overline{\mathrm{G}\langle F\rangle}=\{U\}$. Implying according to Lemma \ref{lem:justonecopy} that $\overline{\mathrm{G}\langle E\rangle}=\{U\}$ for all finite subsets $E$ of $U$. Hence $\mathfrak{rc}(E)=\mathfrak{ic}(E)=\bigcap\overline{\mathrm{G}\langle E\rangle}[U]=\bigcap\{U\}=U$ for all finite subsets $E$ of $U$. 
 \end{proof}
 
 \begin{cor}\label{cor:oenonfinempt} 
  Let $U$ be a countable set and $\mathrm{G}$ be a subgroup of $\mathfrak S(U)$ with $\mathfrak{rc}(\emptyset)\not=U$. Let $F$ be a finite subset of $U$.  Then $C:=\mathfrak{rc}(F)\not\in \overG[U]$. 
 \end{cor}
 \begin{proof}
If $C\in \overG[U]$ let $f\in \overG$ with $f[U]=C$.  According to  Lemma \ref{lem:contallinvariant} a type $\langle F\tr p\rangle$ is ranked if and only if the type $\langle F\tr p\rangle_C$ is ranked with respect to the group action $f[\mathrm{G}]$ on $C$. It follows then from Lemma~\ref{lem:oenonfinempt} that $\mathfrak{rc}(\emptyset)=C$ under the group action $f[\mathrm{G}]$ on $C$.  Hence, every orbit of the group $f[G]$ is ranked.  Which according to   Lemma~\ref{lem:contallinvariant} implies that every orbit of $\mathrm{G}$ is ranked and hence that $\mathfrak{rc}(\emptyset)=U$.

 \end{proof}

 \begin{thm}\label{thm:interschcop}
  Let $U$ be a countable set and $\mathrm{G}$ be a subgroup of $\mathfrak S(U)$ and $\mathfrak{rc}(\emptyset)\not=U$.  Then for every finite subset $F$ of $U$ and every copy $C\in \overG[U]$ with $F\subseteq C$ (hence $\mathfrak{rc}(F)\subseteq C$),  there exists: \\
  An $\omega$-sequence  
  $C=C_0\supset C_1\supset C_2\supset C_3\supset\dots$ of copies $C_i\in \overG[U]$ for which $\mathfrak{rc}(F)=\bigcap_{i\in \omega}C_i$.   
 \end{thm}
 \begin{proof}
Let $F$ be a finite subset of $U$ and let $C_0:=C$. It follows from Lemma~\ref{lem:oenonfinempt} that $\mathfrak{rc}(F)\not=C$. Implying that there exists an unranked type with sockel $F$. Let $P:=\{p_j\mid j\in \omega\}$ be an $\omega$-enumeration of the union $P$ of all the typesets $\mathrm{G}\langle F\tr p\rangle\cap C$ with sockel $F$  for unranked types $\langle F\tr p\rangle$. If  for a type $\langle F\tr p\rangle$ the typeset $\mathrm{G}\langle F\tr p\rangle \cap C$ is finite then the rank of the type $\langle F\tr p\rangle_C$ with respect to the copy $C$ is 0. It follows from Lemma \ref{lem:contallinvariant} that then the  type $\langle F\tr p\rangle$ is ranked. Hence the set $P$ is infinite. 

Assume now that we have already  constructed a sequence $C_0\supset C_1\supset C_2\supset\dots\supset C_{n-1}\supseteq \mathfrak{rc}(F)$ of copies. Then $C_{n-1}\not= \mathfrak{rc}(F)$ according to Corollary \ref{cor:oenonfinempt}.  Let $f\in \overG$ with $f[U]=C_{n-1}$.   Let $m$ be the smallest index for which $p_m\in C_{n-1}$. The type $\langle F\tr p_m\rangle$ is unranked for the group action $\mathrm{G}$. Implying,  according to Lemma \ref{lem:contallinvariant} that the type $\langle F\tr p_m\rangle_{C_{n-1}}$ is unranked with respect to the group action $f[\mathrm{G}]$. Implying that there exists a copy $C_n\subset C_{n-1}$   with $F\subseteq C_n\in \overline{(f[\mathrm{G}])}[C_{n-1}]$. It follows from Lemma~\ref{lem:invarcop} Item (5) that $C_n\in \overG[U]$.  Because $F\subseteq C_n$ it follows from Lemma \ref{lem:contall} that $\mathfrak{rc}(F)\subseteq C_n$. 
 
 \end{proof} 
 
 \begin{cor}\label{cor:interschcop}
 Let $U$ be a countable set and $\mathrm{G}$ be a subgroup of $\mathfrak S(U)$.    Then for all  finite subsets $F$ and $E$ of $U$ with $E\cap\mathfrak{rc}(F)=\emptyset$ and for every copy $C$ with $F\subseteq C$:
 \begin{enumerate}
 \item There exists a copy $C'$  with  $C'\cap E = \emptyset$ and with $F\subseteq C'\subseteq C$.
 \item  For every function $f\in \overG\langle F\rangle$:
 \begin{enumerate}
 \item $\mathfrak{rc}(F)= \mathfrak{ic}(F)= \bigcap  (\overG\langle F\rangle)[U]$.
 \item $f[\mathfrak{rc}(F)]=\mathfrak{rc}(F)$. 
 \end{enumerate}
 \end{enumerate}
 \end{cor}
 \begin{proof}
According to  Theorem \ref{thm:rcclinters},  we have  $\mathfrak{ac}(F)=\mathfrak{rc}(F)$ for every finite subset $F\subseteq U$. Hence, in our case, we have $\mathfrak{ac}(F)\cap E= \emptyset$. Applying Corollary \ref{cor:in/out}, we get Item (1). 
 
Item (2a):  $\mathfrak{rc}(F)= \mathfrak{ic}(F)$ according to Theorem \ref{thm:rcclinters}. Since $\mathfrak{ic}(F)= \bigcap  \overline{\mathrm{G}}[U]\llg F\rrg$ according to  Definition \ref{defin:ranktyp} and $\overline{\mathrm{G}}[U]\llg F\rrg=\overline{\mathrm{G}}[U]\langle F\rangle$ according to Lemma \ref{lem:fix-copy} we obtain $\mathfrak{ic}(F)= \bigcap  \overline{\mathrm{G}}[U]\langle F\rangle$. \\
(2b): It follows from Definition \ref{defin:ranktyp} and from  Lemma \ref{lem:maptypef} that  $f[\mathfrak{rc}(F)]\subseteq \mathfrak{rc}(f[F])=\mathfrak{rc}(F)$. The set  $f[U]$ is a copy containing $F=f[F]\subseteq f[\mathfrak{rc}(F)]$. In particular it follows from Lemma \ref{lem:maptypef} that $f$ maps every typeset of the types  with sockel $F$ to a subset of this typeset. Implying, because $U$ is the union of $F$ together with the union of the typesets of the types with sockel $F$,  that if $f$ would map the typeset of a ranked type $\langle F\tr p\rangle$ to a proper subset then $\mathfrak{rc}(F)$ would not be a subset of $f[U]$. But, it follows from Theorem \ref{thm:rcclinters} that $\mathfrak{rc}(F)\subseteq f[U]$.
 \end{proof}

  \begin{lem}\label{lem:locfinacrc}
  If $\mathrm{G}$ is a algebraically finite subgroup of $\mathfrak S(U)$ then $\mathfrak{rc}(S)=\mathfrak{ac}(S)$ for every subset $S$ of $U$. 
  \end{lem}
  \begin{proof}
  It suffices to prove the lemma for finite subsets $F$ of $U$. If $\langle F\tr p\rangle$ is a type of rank 1 then there exists a finite $F'\supseteq F$ so that every continuation $\langle F'\tr q\rangle$ of $\langle F\tr p\rangle$ has rank 0. Leading to the contradiction that $\mathrm{G}\langle F\tr p\rangle \subseteq \bigcup_{q}\mathrm{G}\langle F'\tr q\rangle$ is finite and hence has rank 0. Hence,  the ranked closure of a finite set is the union of typesets of types which have rank 0, that is the union of typesets which are finite and therefore equal to the algebraic closure.  
\end{proof}

\begin{lem}\label{lem:constrpairextbas}
 Let $U$ be a countably infinite set and $\mathrm{G}$ be a algebraically finite subgroup of $\mathfrak S(U)$ with $\mathfrak{ac}(\emptyset)=\emptyset$.  Then there exist two copies $\{C,D\}\subseteq \overG[U]$ with $C\cap D=\emptyset$. 
\end{lem}
\begin{proof}
Let $\{u_i\mid i\in \omega\}$ be an $\omega$-enumeration of the set $U$. For $n\in \omega$ let $U_n=\{u_i\mid i\in n\}$.    Let $C_0=D_0=U$ and $V_0=W_0=\emptyset$. Let $g_1$ be the identity map on $U$.  For  $V_1:=\{u_0\}$ the algebraic closure of $V_1$ is finite. Hence there exists, according to Corollary \ref{cor:interschcop} Item~(1) with $E$ for $\mathfrak{ac}(V_1)$ and $F$ for $\emptyset$ and $C$ for $U$, a copy $D_1$ with  $\mathfrak{ac}(V_1) \cap D_1=\emptyset$.  The typeset $\mathrm{G}\langle \emptyset\tr u_0\rangle$ of the type $\langle \emptyset\tr u_0\rangle$ is infinite because $\mathfrak{ac}(\emptyset)=\emptyset$. Implying according to Lemma  \ref{lem:contallinvariant} that the typeset $\mathrm{G}\langle \emptyset\tr u_0\rangle\cap D_1$ of the type $\langle \emptyset\tr u_0\rangle_{D_1}$ is infinite.  Let $w_0\in(\mathrm{G}\langle \emptyset\tr u_0\rangle\cap D_1)\setminus \mathfrak{ac}(V_1)$ and let $W_1=\{w_0\}$. Note that there exists a function $h_1\in \mathrm{G}$ with $h_1(u_0)=w_0$.    It follows from Theorem   \ref{thm:rcclinters} that $\mathfrak{ac}(W_1)\subseteq D_1$. Hence $\mathfrak{ac}(V_1)\cap (\mathfrak{ac}(W_1)=\emptyset$. It follows using Corollary \ref{cor:interschcop} Item (1) that there exists a copy $C_1$ with $\mathfrak{ac}(V_1)\subseteq C_1\subseteq U$ and with $\mathfrak{ac}(W_1) \cap C_1=\emptyset$. Then $\mathfrak{ac}(V_1)\subseteq C_1\setminus D_1$ and $\mathfrak{ac}(W_1)\subseteq D_1\setminus C_1$. 

 For $n\in \omega$ assume that there are  copies $C_n$ and $D_n$  and sets $V_n=\{v_i\mid i\in n\}$ with $\mathfrak{ac}(V_n)\subseteq C_n\setminus D_n$ and $W_n=\{w_i\mid i\in n\}$ with $\mathfrak{ac}(W_n)\subseteq D_n\setminus C_n$.  Assume further  that there are functions $g_n$ and $h_n$ in $\mathrm{G}$ with $g_n(u_i)=v_i$ for all $i\in n$ and $h_n(u_i)=w_i$ for all $i\in n$. Note that $\mathfrak{ac}(V_n)\cap \mathfrak{ac}(W_n)=\emptyset$ and that the set $\mathfrak{ac}(V_n)\cup \mathfrak{ac}(W_n)$ is finite.  If the typeset of the type $\langle V_n\tr g(u_n)\rangle$ is finite then it is a subset of the algebraic closure of $V_n$ and hence  a subset of $C_n\setminus D_n$. In this case let $v_n=g_n(u_n)$ and $g_{n+1}=g_n$. If the typeset of the type $\langle V_n\tr g(u_n)\rangle$ is infinite then the typeset of the type  $\langle V_n\tr g(u_n)\rangle_{C_n}$ is infinite. Implying that there exists an element $v_n\in (\mathrm{G}\langle V_n\tr g(u_n)\rangle\cap C_n)\setminus \mathfrak{ac}(W_n)$. Let $V_{n+1}=V_n\cup \{v_n\}$. There exists a function $f\in \mathrm{G}\langle V_n\rangle$ with $f(g_n(u_n))=v_n$. Let $g_{n+1}=f\circ g_n$.  According to Corollary~\ref{cor:interschcop} Item (1) there exists a copy $C_{n+1}$ with $\mathfrak{ac}(V_{n+1})\subseteq C_{n+1}$ and with $C_{n+1}\cap (\mathfrak{ac}(W_n)=\emptyset$.

 If the typeset of the type $\langle W_n\tr h(u_n)\rangle$ is finite then it is a subset of the algebraic closure of $W_n$ and hence  a subset of $ D_n\setminus C_{n+1}$. In this case let $w_n=h_n(u_n)$ and $h_{n+1}=h_n$. If the typeset of the type $\langle W_n\tr h(u_n)\rangle$ is infinite then the typeset of the type  $\langle W_n\tr h(u_n)\rangle_{D_n}$ is infinite. Implying that there exists an element $w_n\in (\mathrm{G}\langle W_n\tr h(u_n)\rangle\cap D_n)\setminus \mathfrak{ac}(C_{n+1})$. Let $W_{n+1}=W_n\cup \{w_n\}$. There exists a function $f'\in \mathrm{G}\langle W_n\rangle$ with $f'(h_n(u_n))=w_n$. Let $h_{n+1}=f'\circ h_n$.  According to Corollary~ \ref{cor:interschcop} Item (1) there exists a copy $D_{n+1}$ with $\mathfrak{ac}(W_{n+1})\subseteq D_{n+1}$ and with $D_{n+1}\cap \mathfrak{ac}(V_{n+1})=\emptyset$. Then $(\mathfrak{ac}(V_{n+1})\subseteq C_{n+1}\setminus D_{n+1}$ and  $\mathfrak{ac}(W_{n+1})\subseteq D_{n+1}\setminus C_{n+1}$.

Recursively for every $n\in \omega$ there exists a set $V_n=\{v_i\mid i\in n\}$ and a set $W_n=\{w_i\mid i\in n\}$ for which $V_n\cap W_n=\emptyset$ and there exists a function $g_n\in \mathrm{G}$ and a function $h_n\in \mathrm{G}$ with $g_n(u_i)=v_i$ for all $i\in n$  and with $h_n(u_i)=w_i$ for all $i\in n$. Let $V=\bigcup_{n\in\omega}V_n$ and let $W=\bigcup_{n\in \omega}W_n$. Then $V\cap W=\emptyset$. Let $g: U\to V$ be the function with $g(u_i)=v_i$ for all $i\in \omega$. Then $g\in \overG[U]$ implying that $V$ is a copy.  Let $h: U\to W$ be the function with $h(u_i)=w_i$ for all $i\in \omega$. Then $h\in \overG[U]$ implying that $W$ is a copy. 
\end{proof}

\begin{thm}\label{thm:constrpairextbas}
 Let $U$ be a countable  set and $\mathrm{G}$ be a algebraically finite subgroup of $\mathfrak S(U)$ and $F$ a finite subset of $U$.    Then there exist two copies $\{C,D\}\subseteq \overG[U]$ with $C\cap D=\mathfrak{ac}(F)$.  
\end{thm}
\begin{proof} 
It follows from Lemma \ref{lem:fix-copy} that every copy containing the finite set $\mathfrak{ac}(F)$ is a copy of the group $\mathrm{G}\langle \mathfrak{ac}(F)\rangle$. Let $V=U\setminus(\mathfrak{ac}(F))$.  Every copy $I$ of the group $\mathrm{G}\langle \mathfrak{ac}(F)\rangle$ is the disjoint union of $\mathfrak{ac}(F)$ and $I\cap V$. The set  $\mathrm{H}=\{\restrict{g}{V}\mid g\in \mathrm{G}\}$ of functions is a subgroup of the symmetric $\mathfrak{S}(V)$ and the function $\sigma: \mathrm{G}\to \mathrm{H}$ with $\sigma(g)=\restrict{g}{V}$ is a group action automorphism. A subset $I$ of $U$ is a copy for the group $\mathrm{G}$ if and only if $I\cap V$ is a copy for the group $\mathrm{H}$. Because $\mathfrak{ac}_{\mathrm{H}}(\emptyset)=\emptyset$ the Theorem follows from Lemma \ref{lem:constrpairextbas}. 
\end{proof}

 \section{The lattice of intersections of copies}

 Let $U$ be a  set. A \emph{Moore family} on   $U$ is a collection $\mathcal M$ of subsets of $U$  which is closed under intersections. The map  $F\rightarrow \phi_{\mathcal M}(F):= \bigcap \{ X\in \mathcal M: F\subseteq X\}$ is the \emph{closure operator} associated with the  Moore family. This closure operator is  \emph{algebraic} if it is closed under union of chains; equivalently, the Moore family is topologically closed in $\powerset(U)$ equipped with the product topology. If a closure operator is algebraic, the Moore family, once ordered by inclusion,  forms an \emph{algebraic lattice}, i.e. a complete lattice in which every element is a join(possibly infinite) of compact elements,  an element $x$ being \emph{compact} if $x\leq \bigvee X$ in the lattice $\mathcal M$ implies $x\leq \bigvee X'$ for some finite subset of $X$ (but the converse is false in general)(see \cite{gratzer} for information about algebraic lattices).  Let $\mathrm{G}$ be a subgroup of $\mathfrak S(U)$ and $\overline{\mathrm{G}}[U]$  the set of copies. 
We denote by $\widehat {\overline{\mathrm{G}}[U]}$ the set of intersections of members of $\overline{\mathrm{G}}[U]$. Being   closed under intersections this set is a \emph{Moore family}; once ordered by inclusion, this family of intersections is a complete lattice. The map $F \rightarrow \mathfrak{ic}(F)$ is the closure operator associated with the above Moore family. If the closure $F \rightarrow \mathfrak{ic}(F)$ is algebraic then $\widehat {\overline{\mathrm{G}}[U]}$ is  an algebraic lattice. We do not know if the converse holds. In fact:

\begin{problem}  
Is $\widehat {\overline{\mathrm{G}}[U]}$ an algebraic lattice? Is the closure $F \rightarrow \mathfrak{ic}(F)$ algebraic?
\end{problem}

 Let us recall that an element $u$ of a poset $P$ is \emph{completely meet-irreducible} if the poset  $P$ contains an element,  denoted $u+$,  such that $u^+>u$ and  $y\geq u^+$ for all $y>u$, with $y\in P$. 
 
 \begin{lem}\label{meet-irreducible} If $U$ is countable,  the following properties are equivalent for  every element $C\in \widehat {\overline{\mathrm{G}}[U]}$:
 \begin{enumerate}
 \item[(i)] $C$  is completely meet-irreducible;
  \item[(ii)] $C\in{\overline{\mathrm{G}}[U]}$ and there is an element  $x\in U\setminus C$ such that $C$ is maximal among the set of copies  ${\overline{\mathrm{G}}[U]}$ not containing $x$. 
 \end{enumerate}
 \end{lem}
 
 \begin{proof} Let $C$ be a completely meet-irreducible element of $\widehat {\overline{\mathrm{G}}[U]}$.  Then trivially $C\in {\overline{\mathrm{G}}[U]}$. Let $x\in U\setminus C$. According to Corollary \ref{cor:maximal} there is some maximal member of ${\overline{\mathrm{G}}[U]}$ which contains $C$ and not $x$.  For each $x\in U\setminus C$,  let $C_x$ be such an element. Then clearly, $C= \bigcap\{C_x: x\in U\setminus C\}$. Since $C$ is completely meet-irreducible then $C=C_x$ for some $x\in U\setminus C$. Conversely, set $C^+:= \{J\in {\overline{\mathrm{G}}[U]} : \{x\}\cup C \subseteq J\}$. Since, by hypothesis  every $J\in {\overline{\mathrm{G}}[U]}$ with  $C\subset J$ will contain $x$, it follows that $C^+\subseteq J$. Hence $C$ is completely meet-irreducible. 
 \end{proof}
 
 \begin{cor} If $U$ is countable, every member of  $\widehat {\overline{\mathrm{G}}[U]}$ is an intersection of completely meet-irreducible elements. 
 \end{cor}
 
 \begin{proof} Let $J\in \widehat {\overline{\mathrm{G}}[U]}$. For every $x\in U\setminus J$ there is some $C'\in \overline{\mathrm{G}}[U]$ such that $J\subseteq C'$ and $x\in U\setminus C'$. According to Lemma \ref{meet-irreducible} above every  $C$ maximal among the  members of ${\overline{\mathrm{G}}[U]}$ containing $C'$ and not containing $x$ is completely meet-irreducible. The intersection of those $C$ is $J$. 
 \end{proof}

\begin{thm}\label{thm:alrcinterscl}
Let $U$ be a countable  set and $\mathrm{G}$ be a  subgroup of $\mathfrak S(U)$. Then for every subset $S$ of $U$, the ranked closure $\mathfrak{rc}(S)$ of $S$, is equal to   the least element of $\overline{\overG[U]}$ containing $S$, where  
$\overline {\overG[U]}$  is the topological closure of $\overG[U]$ in $\powerset (U)$ for  the powerset topology .  
\end{thm}

\begin{proof}
According to Lemma \ref{lem:equalityclosures},   $\overline{\overG[U]}=\overG^{\mathfrak{rc}}[U]$, thus $\overline{\overG[U]}$ is closed under intersection, and the closure of $S$ w.r.t. this closure system is equal to the closure of $S$ w.r.t. $\overG^{\mathfrak{rc}}[U]$. 
\end{proof}

\begin{lem}\label{lem:alrcinterscl}
Let $U$ be a countable  set and $\mathrm{G}$ be a  subgroup of $\mathfrak S(U)$. Then the closure operator associated with $\widehat{\overG[U]}$ is  algebraic  if and only if $\widehat{\overG[U]}=\overline{\overG[U]}$. 
\end{lem}
\begin{proof}
The Moore family $\widehat{\overG[U]}$ defines  an algebraic closure system if and only if  it is closed in $\powerset(U)$ equipped with the product topology. Since $\overG[U] \subseteq 
\widehat{\overG[U]} \subseteq \overline {\overG[U]}$
 by Lemma \ref{lem:equalityclosures},  it is closed if and only if it is equal to $\overline {\overG[U]}$. 
\end{proof}
     
\begin{problem} 
Is $\mathfrak{rc}(S)=\mathfrak{ic}(S)$ for all subsets $S$ of $U$?
\end{problem}

\begin{problem} Suppose that there is more than one copy. Does then $U$ contain an infinite independent set w.r.t the closure system given by the collection of intersections of copies.  That is an infinite subset $S$ such that $x\not \in \mathfrak{ic}({S\setminus \{x\}})$ for every $x\in S$?  Does this conclusion hold under  the additional assumption that the algebraic closure of every finite set is finite? (The group is then algebraically finite.)  
\end{problem}

\section{Embedding the powerset of $\aleph_0$}

\begin{thm}\label{thm:Woligomorphic}
Let $\mathfrak{M}:=(M,R_i)_{i\in I}$ be an infinite countable relational structure with language $\mathcal{L}$ in which 
$T:=Th(\mathfrak{M})$ is $\aleph_0$-categorical. Let $\mathbb{Q}$ be the set of 
rational numbers and let $\mathscr{B}$ be the lattice of subsets of $\mathbb{Q}$
ordered under set inclusion. Let $\mathscr{C}$ be the family of subsets $N\subseteq M$
for which $\restrict{\mathfrak{M}}{N}$, the {\em restriction} of $\mathfrak{M}$ to $N$,
is isomorphic to $\mathfrak{M}$. (i.e., $\mathscr {C}$ is the set of copies.) Then for $\mathscr{C}$ ordered by
 set inclusion there exists an  order embedding of $\mathscr{B}$ into 
$\mathscr{C}$.
\end{thm}
\begin{proof}
\underline{Step 1:} Let $\mathfrak{M}_0:=\mathfrak{M}$ with language $\mathcal{L}_0:=\mathcal{L}$. Add Skolem functions expanding $\mathfrak{M}$ to a structure
$\mathfrak{M}'$. This is done iteratively. Fix an element, say $\nu\in M$. 
For each $L_0$ formula $\exists x \varphi(x,y_1,\dots,y_n)$, with at most $y_1,\dots,y_n$
occurring freely, add a function $f_{\exists x\varphi}: M^n\to M$ and axiom:
\[
\exists x\varphi(x,y_1,\dots,y_n) \leftrightarrow \varphi[f_{\exists x\varphi|}(y_1,\dots,y_n), y_1,\dots,y_n].
\]
Expand $\mathfrak{M}_0$ to $\mathfrak{M}_1$ using the axiom of choice to pick a witness
when $\exists x \varphi(x,y_1,\dots, y_n)$ holds for $f_{\exists x\varphi|}(y_1,\dots,y_n)$ and define 
$f_{\exists x\varphi}(y_1,\dots,y_n)=\nu$ otherwise. This gives the language $\mathcal{L}_1$ and the expansion $\mathfrak{M}_1$ of $\mathfrak{M}_0$. Iterate to form $\mathcal{L}_k\subseteq \mathcal{L}_{k+1}$ and $\mathfrak{M}_{k+1}$ an expansion of $\mathfrak{M}_k$. Then $\mathcal{L}'$ is the union of the $\mathcal{L}_k$
and $\mathfrak{M}'$ the common expansion of the structures $\mathfrak{M}_k$ in the 
natural way. Let $T'$ be the theory of $\mathfrak{M}'$.

It is a standard fact that if $\mathcal{S}$ with domain $S$ is a substructure, (in the expanded language,) of $\mathfrak{M}'$ then it is an elementary substructure of $\mathfrak{M}'$. (It is easy to check that if $\{n_1,n_2,\dots, n_k\}\subseteq S$ and if $\exists x\varphi(x,n_1,n_2,\dots,n_k)$ is a sentence true in $\mathfrak{M}'$ then $f_{\exists x\varphi}(n_i\dots n_k)\in S$  and $\mathfrak{M}' \models   \varphi[f_{\exists x\varphi}(n_1,\dots,n_k),n_1,\dots,n_k]$.)

\vskip 4pt
\noindent
\underline{Step 2:} 
We use Theorem 3.3.10 of \cite{Chang-Keisler}: 

Theorem: \label{thm:ChKei} 
{\em Let $T$ be a theory in $\mathcal{L}$ with infinite models, and let $\langle X,<\rangle$ be a simply ordered set. Then there is a model $\mathfrak{A}$ of $T$ with $X$ subset of the domain of $\mathfrak{A}$ and such that $X$ is a set of indiscernibles in $\mathfrak{A}$.}

Indiscernibles with respect to  Theorem \ref{thm:ChKei} in \cite{Chang-Keisler} means order indiscernibles. 
Translated to our notation we obtain:

Theorem:\label{thm:LPSW}
{\em There is a model $\mathfrak{M}''$ of $\mathcal{L}'$ with $\mathbb{Q}$ subset of the domain $M''$ of  $\mathfrak{M}''$ and such that $\mathbb{Q}$ is a set of indiscernibles in $\mathfrak{A}$.}

Let $\mathfrak{A}$ be the reduct of $\mathfrak{M}''$ to $\mathcal{L}$. Because the structure $\mathfrak{M}$ is $\aleph_0$-categorical we may assume without loss of generality that $\mathfrak{M}=\mathfrak{A}$. 

For $X\subset \mathbb{Q}$ let $\mathfrak{M}'_X$ be the closure of $X$ under Skolem functions. 
  Because $\mathfrak{M}'_X$ is an elementary restriction of $\mathfrak{M}'$,  the theory of $\mathfrak{M}'_X$ is $T'$ and hence  the reduct to $\mathcal{L}$ is a model whose theory is $T$. Implying, again due to categoricity,  that   $\mathfrak{M}'_X$ is a copy of $\mathfrak{M}$. Note that $\nu\in \mathfrak{M}'_X$ only if there is a term $f_{\exists x\varphi}(x_1,\dots,x_k)$ for which $\exists x\varphi(x,x_1,\dots,x_k)$ holds.  Also, it is clear that $X\subseteq X_1$ implies $\mathfrak{M}'_X\subseteq \mathfrak{M}'_{X_1}$. It remains to prove that if $X\not= X_1$ then  $\mathfrak{M}'_X\not= \mathfrak{M}'_{X_1}$. This will follow from the following: 

\vskip 3pt
\noindent
\underline{Claim:} Let $X\subseteq \mathbb Q$ then $\mathfrak{M}'_X\cap X=X$. \\ 
Let $y\not\in X$ but $y\in \mathfrak{M}'_X$, for a contradiction. Then there is a term $t$ of $\mathcal{L}'$ and rationals $x_1<\dots<x_k$ in $X$ so that $y=t(x_1,\dots, x_k)$. According to construction $\mathrm{M}'\models \exists ! x\, (x=t(x_1,\dots,x_k))$. But if a rational $y$ is in the same interval of $x_1,\dots, x_k$ as $x$ then $\mathfrak{M}'\models y=t(x_1,\dots,x_k)$, because the rationals form a set of order indiscernables. 

 \end{proof}

 \begin{thm}\label{thm:Woligomorphic2}
Let $U$ be a countable set and $\mathrm{G}$ be  an oligomorphic subgroup of $\mathfrak{S}(U)$. Then  there exists an embedding of $(\powerset (\omega);\subseteq)$, the power set of $\omega$ ordered by  inclusion,  into  $(\overline{\mathrm{G}}[U];\subseteq)$, the set of copies for the group $\mathrm{G}$ ordered by inclusion.
\end{thm}
\begin{proof}
Let $\mathfrak{M}$ be the homogeneous canonical relational structure associated with the group $\mathrm{G}$. (See \cite{cameronbook} page 26.) Then $\mathrm{G}=\aut(\mathfrak{M})$.  It follows from \cite{ryll}  that the structure $\mathfrak{M}$ is $\aleph_0$-categorical.  Let $\mathscr{C}$ be the family of subsets $N\subseteq M$
for which $\restrict{\mathfrak{M}}{N}$, the {\em restriction} of $\mathfrak{M}$ to $N$,
is isomorphic to $\mathfrak{M}$. Then $\mathscr{C}$ is the set of copies of $\mathfrak{M}$. The structure $\mathfrak{M}$ is homogeneous implying that $\mathscr{C}$ is also the set of copies for  the group $\mathrm{G}$, that is $\mathscr{C}=\overG[U]$.  The Theorem follows from Theorem \ref{thm:Woligomorphic}. 
 
 \end{proof}

\end{document}